\documentclass[11pt,a4paper]{article}

\usepackage{epsf,epsfig,amsfonts,amsgen,amsmath,amstext,amsbsy,amsopn,amsthm,cases,listings,color
}
\usepackage{ebezier,eepic}
\usepackage{color}
\usepackage{multirow}
\usepackage{epstopdf}
\usepackage{graphicx}   
\usepackage{pgf,tikz}
\usepackage{mathrsfs}
\usepackage[marginal]{footmisc}
\usepackage{enumerate}
\usepackage{enumitem}
\usepackage[titletoc]{appendix}
\usepackage{booktabs}
\usepackage{url}
\usepackage{mathtools}

\usepackage{pgfplots}
\usepackage{authblk}
\usepackage{amssymb}

\usepackage{wasysym}

\usepackage{empheq}

\usepackage{dsfont}

\usepackage{tikz}
\usepackage{longtable}
\usepackage{float}
\usepackage{subfigure}
\pgfplotsset{compat=1.18}
\usepackage{mathrsfs}
\usepackage{wasysym} 
\usetikzlibrary{arrows}
\usepackage{aligned-overset}
\usepackage{bm}
\usepackage{bbm}
\usepackage[T1]{fontenc}
\usepackage[backref=page]{hyperref}

\usepackage{algorithm}
\usepackage{algpseudocode}
\usepackage{multicol}

\usepackage[thinc]{esdiff}

\allowdisplaybreaks[1]

\definecolor{uuuuuu}{rgb}{0.27,0.27,0.27}
\definecolor{sqsqsq}{rgb}{0.1255,0.1255,0.1255}

\newtheorem{definition}{Definition} [section]
\newtheorem{theorem}[definition]{Theorem}
\newtheorem{lemma}[definition]{Lemma}

\newtheorem{claim}[definition]{Claim}

\newtheorem{fact}[definition]{Fact}

\newcommand{\norm}[1]{\lVert#1\rVert}

\usepackage{titlesec} 

\begin{document}
\title{\bf\Large Tetrahedron Conjecture in the $\ell_2$-norm}
\author{
Levente~Bodn\'ar \quad 
Wanfang~Chen \quad 
Jinghua~Deng \quad 
Jianfeng~Hou \\\vspace{-0.5em} 
Xizhi~Liu  \quad
Jialei~Song \quad
Jiabao~Yang \quad 
Yixiao~Zhang
\vspace{-1.5em} 
}
\date{\today}
\maketitle
\begin{abstract}
    The famous Tetrahedron Conjecture of Tur\'{a}n from the 1940s asserts that the number of edges in an $n$-vertex $3$-graph without the tetrahedron, the complete $3$-graph on four vertices, cannot exceed that of the balanced complete cyclic $3$-partite $3$-graph, whose edges are of types $V_1 V_2 V_3$, $V_1 V_1 V_2$, $V_2 V_2 V_3$, and $V_3V_3 V_1$.
    A recent surprising result of Balogh--Clemen--Lidick\'{y}~[J. Lond. Math. Soc. (2) 106 (2022)] shows that this conjecture is asymptotically true in the $\ell_2$-norm, where the number of edges is replaced by the sum of squared codegrees. They further conjectured that, in this $\ell_2$-norm setting, the $3$-partite construction is uniquely extremal for large $n$. 
    We confirm this conjecture. 
    
    Two key ingredients in our proofs include establishing a Mantel theorem for vertex-colored graphs that forbid certain types of triangles, and introducing a novel procedure integrated into Simonovits’ stability method, which essentially reduces the task to verifying that the $\ell_2$-norm of certain near-extremal constructions increases under suitable local modifications. The strategy in the latter may be of independent interest and potentially applicable to other extremal problems. 
\end{abstract}

\renewcommand\thefootnote{} 
\addtocounter{footnote}{-1} 
\renewcommand\thefootnote{\arabic{footnote}} 

\section{Introduction}\label{SEC:Introduction}
Given an integer $r\ge 2$, an \emph{$r$-uniform hypergraph} (henceforth an \emph{$r$-graph}) $\mathcal{H}$ is a collection of $r$-subsets of some set $V$. 
We call $V$ the \emph{vertex set} of $\mathcal{H}$ and denote it by $V(\mathcal{H})$. 
When $V$ is understood, we usually identify a hypergraph $\mathcal{H}$ with its set of edges.

Given a family $\mathcal{F}$ of $r$-graphs, an $r$-graph $\mathcal{H}$ is \emph{$\mathcal{F}$-free} if it does not contain any member of $\mathcal{F}$ as a subgraph. 
A fundamental problem in Extremal Combinatorics is to determine the extremal properties of $\mathcal{F}$-free $r$-graphs on $n$ vertices. 
In particular, the Tur\'{a}n problem asks for the \emph{Tur\'{a}n number} $\mathrm{ex}(n,\mathcal{F})$, the maximum number of edges in an $\mathcal{F}$-free $r$-graphs on $n$ vertices. 

For $r=2$ (i.e., graphs), the value of $\mathrm{ex}(n,\mathcal{F})$ is well understood up to an additive error term of order $o(n^2)$ thanks to the celebrated Erd\H{o}s--Stone Theorem~\cite{ES46} (see also~\cite{ES66}). 
In contrast, for $r \ge 3$, determining $\mathrm{ex}(n,\mathcal{F})$ even asymptotically is notoriously difficult in general. 
Denote by $K_{\ell}^{r}$ the complete $r$-graph on $\ell$ vertices (with the superscript omitted when $r=2$). 
The seminal work of Tur\'{a}n~\cite{Tur41} determined the exact value of $\mathrm{ex}(n,K_{\ell})$ for every $\ell \ge 3$ (with the case $\ell=3$ solved earlier by Mantel~\cite{Mantel07}). 
Erd\H{o}s~\cite{Erd81} offered \$500 for the asymptotic determination of $\mathrm{ex}(n,K_{\ell}^{r})$ for any single pair $\ell, r$ with  $\ell > r \ge 3$. 
This prize remains unclaimed, despite decades of intensive work. For further background and related results, we refer the reader to the surveys~\cite{Fur91,Caen94Survey,Sid95,Kee11}. 

Perhaps the most famous instance is the tetrahedron $K_{4}^{3}$, the complete $3$-graph on four vertices. 
Given three pairwise disjoint sets $V_1, V_2, V_3$, let $\mathbb{C}[V_1, V_2, V_3]$ denote the $3$-graph on $V_1 \cup V_2 \cup V_3$ consisting of all triples of types $V_1 V_2 V_3$, $V_1 V_1 V_2$, $V_2 V_2 V_3$, and $V_3 V_3 V_1$ (see Figure~\ref{fig:CB}), that is, 
\begin{align*}
    \mathbb{C}[V_1, V_2, V_3]
    \coloneqq \left\{e \colon \big(|e\cap V_1|, |e\cap V_2|, |e\cap V_3|\big) \in \{(1,1,1), (2,1,0),(0,2,1), (1,0,2)\}\right\}. 
\end{align*}
When $V_1  \cup V_2 \cup V_3 = [n]$ is a balanced partition of $[n]$, we write $\mathbb{C}_{n}$ for $\mathbb{C}[V_1, V_2, V_3]$.

The famous Tetrahedron Conjecture of Tur\'{a}n~\cite{Tur41} asserts that $\mathrm{ex}(n,K_{4}^{3}) = |\mathbb{C}_{n}|$. 
Various distinct constructions achieving the same bound as $\mathbb{C}_{n}$ were discovered later (see, e.g.,~\cite{Bro83,Kos82,Fon88,Fro08}). 
In particular, this implies that the Tur\'{a}n problem for $K_{4}^{3}$ does not exhibit the stability property (see~{\cite[Proposition~1.7]{LM22stab}}), which is likely one reason for the notorious difficulty of the conjecture. 
Successively better upper bounds for $\mathrm{ex}(n,K_{4}^{3})$ have been obtained by many researchers (see e.g.,~\cite{Caen88,CL99,Razborov10,BT11}), with the current record established by the computer-assisted flag algebra framework of Razborov~\cite{Raz07}.

\begin{figure}[H]
\centering

\tikzset{every picture/.style={line width=1pt}} 

\begin{tikzpicture}[x=0.75pt,y=0.75pt,yscale=-1,xscale=1,line join=round, scale=0.8]

\draw   (152,80.5) .. controls (152,58.13) and (170.13,40) .. (192.5,40) .. controls (214.87,40) and (233,58.13) .. (233,80.5) .. controls (233,102.87) and (214.87,121) .. (192.5,121) .. controls (170.13,121) and (152,102.87) .. (152,80.5) -- cycle ;
\draw   (88.25,186.5) .. controls (88.25,164.13) and (106.38,146) .. (128.75,146) .. controls (151.12,146) and (169.25,164.13) .. (169.25,186.5) .. controls (169.25,208.87) and (151.12,227) .. (128.75,227) .. controls (106.38,227) and (88.25,208.87) .. (88.25,186.5) -- cycle ;
\draw   (215.75,186.5) .. controls (215.75,164.13) and (233.88,146) .. (256.25,146) .. controls (278.62,146) and (296.75,164.13) .. (296.75,186.5) .. controls (296.75,208.87) and (278.62,227) .. (256.25,227) .. controls (233.88,227) and (215.75,208.87) .. (215.75,186.5) -- cycle ;
\draw[fill=uuuuuu, fill opacity=0.3]   (191.5,101) -- (237.5,177) -- (145.5,177) -- cycle ;
\draw [fill=uuuuuu] (191.5,101) circle (1.2pt);
\draw [fill=uuuuuu]  (237.5,177) circle (1.2pt);
\draw [fill=uuuuuu] (145.5,177) circle (1.2pt);
%
\draw[fill=uuuuuu, fill opacity=0.3]    (217.5,67) .. controls (219.5,99) and (246.5,145) .. (266.5,184) ..  controls  (238.5,138) and (216.5,98) .. (193.5,85) .. controls (206.5,85) and (215.5,80) .. (217.5,67);
\draw [fill=uuuuuu] (217.5,67) circle (1.2pt);
\draw [fill=uuuuuu]  (266.5,184) circle (1.2pt);
\draw [fill=uuuuuu] (193.5,85) circle (1.2pt);
%
\draw[fill=uuuuuu, fill opacity=0.3]    (253.64,213.54) .. controls (224.11,201.05) and (170.9,204.69) .. (127.07,205.18) .. controls (180.73,200.65) and (226.35,198.81) .. (248.24,184.03) .. controls (242.44,195.66) and (242.9,205.95) .. (253.64,213.54) ;
\draw [fill=uuuuuu] (253.64,213.54) circle (1.2pt);
\draw [fill=uuuuuu]  (127.07,205.18) circle (1.2pt);
\draw [fill=uuuuuu] (248.24,184.03) circle (1.2pt);
%
\draw[fill=uuuuuu, fill opacity=0.3]    (101.86,179.47) .. controls (128.25,161.25) and (153.73,114.39) .. (176.82,77.14) .. controls (151.86,124.86) and (128.94,164.33) ..   (129.66,190.74) .. controls (122.96,179.61) and (114.03,174.48) .. (101.86,179.47) ;
\draw [fill=uuuuuu] (101.86,179.47) circle (1.2pt);
\draw [fill=uuuuuu]  (176.82,77.14) circle (1.2pt);
\draw [fill=uuuuuu] (129.66,190.74) circle (1.2pt);
%
\draw   (388,133) .. controls (388,88.82) and (403.67,53) .. (423,53) .. controls (442.33,53) and (458,88.82) .. (458,133) .. controls (458,177.18) and (442.33,213) .. (423,213) .. controls (403.67,213) and (388,177.18) .. (388,133) -- cycle ;
\draw   (477,133) .. controls (477,88.82) and (492.67,53) .. (512,53) .. controls (531.33,53) and (547,88.82) .. (547,133) .. controls (547,177.18) and (531.33,213) .. (512,213) .. controls (492.67,213) and (477,177.18) .. (477,133) -- cycle ;
\draw [fill=uuuuuu, fill opacity=0.3]   (417.5,97.86) .. controls (441.01,110.72) and (483.64,107.75) .. (518.72,107.8) .. controls (475.72,111.66) and (439.2,112.93)  ..  (421.53,127.43) .. controls (426.28,115.87) and (426.02,105.58) .. (417.5,97.86) ;
\draw [fill=uuuuuu] (417.5,97.86) circle (1.2pt);
\draw [fill=uuuuuu]  (518.72,107.8) circle (1.2pt);
\draw [fill=uuuuuu] (421.53,127.43) circle (1.2pt);
%
\draw [fill=uuuuuu, fill opacity=0.3]   (521.28,169.35) .. controls (497.26,157.46) and (454.8,162.18) .. (419.74,163.56)  .. controls (462.54,157.94) and (498.98,155.18) .. (516.05,139.96) .. controls (511.76,151.71) and (512.45,161.98) .. (521.28,169.35) ;
\draw [fill=uuuuuu] (521.28,169.35) circle (1.2pt);
\draw [fill=uuuuuu]  (419.74,163.56) circle (1.2pt);
\draw [fill=uuuuuu] (516.05,139.96) circle (1.2pt);
%
\draw (182,15) node [anchor=north west][inner sep=0.75pt]   [align=left] {$V_1$};
\draw (250,233) node [anchor=north west][inner sep=0.75pt]   [align=left] {$V_2$};
\draw (115,233) node [anchor=north west][inner sep=0.75pt]   [align=left] {$V_3$};
\draw (415,233) node [anchor=north west][inner sep=0.75pt]   [align=left] {$V_1$};
\draw (508,233) node [anchor=north west][inner sep=0.75pt]   [align=left] {$V_2$};
\end{tikzpicture}

\caption{Structures of $\mathbb{C}[V_1, V_2, V_3]$ and $\mathbb{B}[V_1, V_2]$.}
\label{fig:CB}
\end{figure}
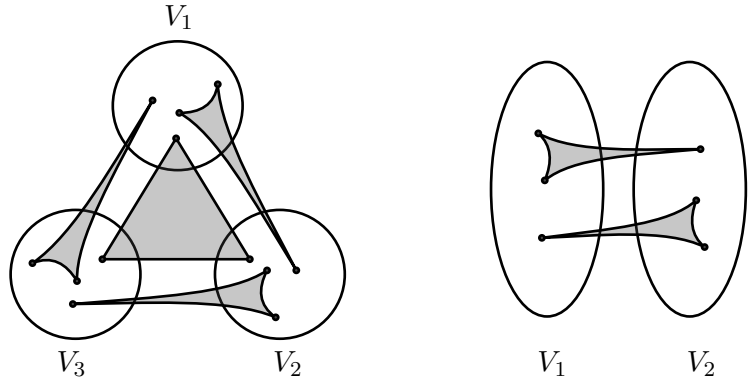

Given an $r$-graph $\mathcal{H}$ on a vertex set $V$, the \emph{codegree} of an $(r-1)$-subset $S \subseteq V$ in $\mathcal{H}$, denoted by $d_{\mathcal{H}}(S)$, is the number of edges in $\mathcal{H}$ containing $S$. 
An simple double-counting argument shows that $\sum_{e\in \binom{V}{r-1}} d_{\mathcal{H}}(e) = r|\mathcal{H}|$. 
Balogh--Clemen--Lidick\'{y} viewed the summation on the left side of the equality as the $\ell_{1}$-norm of $\mathcal{H}$, and they further introduced the $\ell_{2}$-norm of $\mathcal{H}$, given by 
\begin{align*}
    \norm{\mathcal{H}}_{2}
    \coloneqq \sum_{e\in \binom{V}{r-1}} d_{\mathcal{H}}^{2}(e). 
\end{align*}
The Tur\'{a}n number of $\mathcal{F}$ in the $\ell_{2}$-norm is then defined as 
\begin{align*}
    \mathrm{ex}_{\ell_{2}}(n,\mathcal{F})
    \coloneqq \max\left\{\norm{\mathcal{H}}_{2} \colon \text{$v(\mathcal{H}) = n$ and $\mathcal{H}$ is $\mathcal{F}$-free} \right\}. 
\end{align*}
The Tur\'{a}n density of $\mathcal{F}$ in the $\ell_{2}$-norm is given by 
\begin{align*}
    \pi_{\ell_{2}}(\mathcal{F})
    \coloneqq \lim_{n\to \infty} \frac{\mathrm{ex}_{\ell_{2}}(n,\mathcal{F})}{\binom{n}{r-1} (n-r+1)^2}. 
\end{align*}
Tur\'{a}n problems under the $\ell_{2}$-norm were systematically studied by Balogh--Clemen--Lidick\'{y} in~\cite{BCL22a}, and later extended to the more general $(t,p)$-norms as well as to degenerate hypergraphs in~\cite{CILLP24,GLMP24}.

The construction $\mathbb{C}_{n}$ shows that $\pi_{\ell_{2}}(K_{4}^{3}) \ge \frac{1}{3}$. 
Using Razborov's flag algebra framework, Balogh--Clemen--Lidick\'{y}~\cite{BCL22b} proved, with computer assistance, that $\pi_{\ell_{2}}(K_{4}^{3}) \le \frac{1}{3}$ and thus established the following theorem. 

\begin{theorem}[{\cite[Theorem~1.2]{BCL22b}}]\label{THM:BCL22-L2-density}
    We have $\mathrm{ex}_{\ell_{2}}(n, K_{4}^{3}) = \frac{n^4}{6} + o(n^4)$. 
\end{theorem}

They further conjectured (see {\cite[Conjecture~1.3]{BCL22b}}) that for sufficiently large $n$, the construction $\mathbb{C}_{n}$ is the unique extremal configuration for the Tur\'{a}n problem of $K_{4}^{3}$ in the $\ell_{2}$-norm,  in contrast to the $\ell_{1}$-norm setting, where exponentially many extremal constructions may exist. Our main result confirms this conjecture.

\begin{theorem}\label{THM:L2-exact-K43}
    There exists $n_{0}$ such that the following holds for every $n \ge n_{0}$.
    Suppose that $\mathcal{H}$ is an $n$-vertex $K_{4}^{3}$-free $3$-graph. 
    Then $\norm{\mathcal{H}}_{2} \le \norm{\mathbb{C}_{n}}_{2}$, and equality holds iff $\mathcal{H} \cong \mathbb{C}_{n}$. 
\end{theorem}

In Section~\ref{SEC:vtx-colored-Turan}, we prove Tur\'{a}n-type results on vertex-colored graphs without certain types of triangles, in both the $\ell_1$-norm (Theorem~\ref{THM:3-colored-Mantel-L1-norm}) and the $\ell_{2}$-norm (Theorem~\ref{THM:3-colored-Mantel-L2-norm}), together with a corresponding stability theorem (Theorem~\ref{THM:3-colored-Mantel-stability}). 
Although the $\ell_{2}$-norm result can in principle be proved without computer-assisted flag algebra computations, the argument we found requires extremely lengthy case analyses; therefore, we include a flag algebra proof here. 
In Section~\ref{SEC:proof-THM-L2-exact-K43}, we present the proof of Theorem~\ref{THM:L2-exact-K43}. 
The proof uses Simonovits' stability method~\cite{Sim68}, combined with a novel procedure that reduces the problem to the much simpler task of verifying that the $\ell_{2}$-norm of certain near-extremal constructions increases under specific local changes (see Lemmas~\ref{LEMMA:K43-L2-improvement-phase-one} and~\ref{LEMMA:K43-L2-improvement-phase-two}). We believe that this framework applies to other extremal problems as well. 
In Section~\ref{SEC:Remarks}, we include some concluding remarks.

\section{Mantel's theorem in vertex-colored graphs}\label{SEC:vtx-colored-Turan}
In this section, we consider a variation of Mantel’s theorem in vertex-colored graphs. 

Let $G$ be a graph on vertex set $V$, and let $V \coloneqq V_1 \cup V_2 \cup V_3$ be a partition of $V$ (equivalently, a $3$-coloring of the vertices of $G$).
For a triple of vertices $\{u,v,w \} \subseteq V$, we say that $G[\{u,v,w\}]$ is \emph{a triangle of type $V_iV_jV_k$} for some $(i,j,k) \in [3]^3$ if $\{u,v,w\}$ induces a copy of $K_3$ in $G$ and $(u,v,w) \in V_i \times V_j \times V_{k}$.
Denote by $\rho_{3}(G)$ the density of triangles of types $V_1V_2V_3$, $V_1V_1V_2$, $V_2V_2V_3$, and $V_3V_3V_1$ in $G$, that is, the probability that a uniformly chosen $3$-subset of $V$ induces a triangle of one of these types. We say that $G$ is \emph{cyclically triangle-free} if $\rho_{3}(G) = 0$. 

For a vertex set $S \subseteq V$, we use $G[S]$ to denote the \emph{induced subgraph} of $G$ on $S$. 
For pairwise disjoint sets $S_1, \ldots, S_{k} \subseteq V$, we use $G[S_1, \ldots, S_k]$ to denote the \emph{induced $k$-partite subgraph} of $G$ with parts $S_1, \ldots, S_k$. 

Let $\Lambda[V_1, V_2, V_3]$ denote the graph on vertex set $V$ whose edge set is 
\begin{align*}
    \left\{e \in \tbinom{V}{2} \colon (|e\cap V_1|, |e \cap V_2|, |e \cap V_{3}|) \in \{(1,1,0), (0,1,1), (0,0,2)\} \right\}. 
\end{align*}
Observe that $\Lambda[V_1, V_2, V_3]$ is cyclically triangle-free.
Suppose that $|V_1| = |V_2| = |V_3| = n$. Then straightforward calculations show that 
\begin{align*}
    |\Lambda[V_1, V_2, V_3]|
    & = 2n^2 + \binom{n}{2}
    = \frac{5}{2} n^2 - \frac{n}{2}, \\
    \norm{\Lambda[V_1, V_2, V_3]}_{2}
    & = n\cdot n^2 + n \cdot (2n)^2 + n \cdot (2n-1)^2 
    = 9n^3 - 4n^2 + n. 
\end{align*}

Throughout this section, we assume that $G$ is a graph with vertex set $V$, and that $V = V_1 \cup V_2 \cup V_3$ is a fixed partition. We further assume that all indices are taken modulo $3$.

The main results of this section are as follows.
\begin{theorem}\label{THM:3-colored-Mantel-L1-norm}
    For every $\varepsilon > 0$, there exist $\delta_{\ref{THM:3-colored-Mantel-L1-norm}} = \delta_{\ref{THM:3-colored-Mantel-L1-norm}}(\varepsilon) > 0$ and $N_{\ref{THM:3-colored-Mantel-L1-norm}} = N_{\ref{THM:3-colored-Mantel-L1-norm}}(\varepsilon)$ such that the following holds for $n \ge N_{\ref{THM:3-colored-Mantel-L1-norm}}$. 
    Suppose that $\max_{i \in [3]}\left| |V_i| - n \right| \le \delta_{\ref{THM:3-colored-Mantel-L1-norm}} n$ and $\rho_{3}(G) \le \delta_{\ref{THM:3-colored-Mantel-L1-norm}}$. 
    Then 
    \begin{align*}
        |G| \le \frac{5}{2} n^2 + \varepsilon n^2.
    \end{align*}
\end{theorem}

We will also use the following stability theorem. 
\begin{theorem}\label{THM:3-colored-Mantel-stability}
    For every $\varepsilon > 0$, there exist $\delta_{\ref{THM:3-colored-Mantel-stability}} = \delta_{\ref{THM:3-colored-Mantel-stability}}(\varepsilon) > 0$ and $N_{\ref{THM:3-colored-Mantel-stability}}$ such that the following holds for $n \ge N_{\ref{THM:3-colored-Mantel-stability}}$. 
    Suppose that $\max_{i \in [3]}\left| |V_i| - n \right| \le \delta_{\ref{THM:3-colored-Mantel-L2-norm}} n$, $\rho_{3}(G) \le \delta_{\ref{THM:3-colored-Mantel-stability}}$, and 
    \begin{align*}
        |G| + |G[V_1, V_2, V_3]| 
        \ge \frac{9 n^2}{2} - \delta_{\ref{THM:3-colored-Mantel-stability}} n^2. 
    \end{align*}
    Then, there exists $i \in [3]$ such that 
    \begin{align*}
        |G[V_{i}]| + |G[V_{i+1}]| + |G[V_{i}, V_{i+2}]| 
        \le \varepsilon n^2. 
    \end{align*}
\end{theorem}

We say $G$ is \emph{locally maximal} if there exists $i \in [3]$ such that 
\begin{enumerate}[label=(\roman*)]
    \item\label{Def:local-max-1} $|G[V_{i}, V_{i+1}]| + |G[V_{i+2}]| \ge |G[V_{i}, V_{i+2}]| + |G[V_{i}]|$, and 
    \item\label{Def:local-max-2} $|G[V_{i+1}, V_{i+2}]| + |G[V_{i+2}]| \ge |G[V_{i}, V_{i+2}]| + |G[V_{i+1}]|$. 
\end{enumerate}

\begin{theorem}\label{THM:3-colored-Mantel-L2-norm}
    For every $\varepsilon > 0$, there exist $\delta_{\ref{THM:3-colored-Mantel-L2-norm}} = \delta_{\ref{THM:3-colored-Mantel-L2-norm}}(\varepsilon) > 0$ and $N_{\ref{THM:3-colored-Mantel-L2-norm}} = N_{\ref{THM:3-colored-Mantel-L2-norm}}(\varepsilon)$ such that the following holds for $n \ge N_{\ref{THM:3-colored-Mantel-L2-norm}}$. 
    Suppose that $G$ is locally maximal, $\max_{i \in [3]}\left| |V_i| - n \right| \le \delta_{\ref{THM:3-colored-Mantel-L2-norm}} n$, and $\rho_{3}(G) \le \delta_{\ref{THM:3-colored-Mantel-L2-norm}}$. 
    Then 
    \begin{align*}
        \norm{G}_{2} \le 9n^3 + \varepsilon n^3. 
    \end{align*}
\end{theorem}

We remark that the locally maximal assumption is needed only for the flag algebra proof, the alternative proof, which avoids flag algebras, does not require this assumption.

\subsection{Proof of Theorem~\ref{THM:3-colored-Mantel-L1-norm}}\label{SUBSEC:proof-3-colored-Mantel-L1-norm}
In this subsection, we prove Theorem~\ref{THM:3-colored-Mantel-L1-norm}.
Note that by removing vertices or adding isolated vertices to each $V_i$, we may assume that $|V_i| = n$ for all $i \in [3]$.
This modification affects only a negligible number of edges (at most $9\delta_{\ref{THM:3-colored-Mantel-L1-norm}}(1+\delta_{\ref{THM:3-colored-Mantel-L1-norm}})n^2 \ll \varepsilon n^2$ if we choose $\delta_{\ref{THM:3-colored-Mantel-L1-norm}} \ll \varepsilon$).
Combining this observation with the well-known Removal Lemma\footnote[1]{A straightforward modification of the standard proof of the Removal Lemma yields the version we require.} (see e.g.,~\cite{RS78}), we reduce Theorem~\ref{THM:3-colored-Mantel-L1-norm} to the following statement.

\begin{theorem}\label{THM:3-colored-Mantel-L1-norm-reduced}
    Suppose that $|V_i|= n$ for $i \in [3]$ and $G$ is cyclically triangle-free. 
    Then $|G| \le 5n^2/2 + 5n$.
\end{theorem}

For convenience, we assume that $|V_i| = n$ for $i \in [3]$ throughout this subsection. 

We use a slightly modified version of the Zykov symmetrization~\cite{Zyk49} for the proof of Theorem~\ref{THM:3-colored-Mantel-L1-norm-reduced}. 
Given distinct vertices $u,v \subseteq V$, denote by $G_{u \to v}$ the graph obtained from $G$ by \emph{symmetrizing} $u$ to $v$, that is, 
\begin{align*}
    G_{u \to v}
    \coloneqq \big( G \setminus \left\{uw \colon w\in N_{G}(u)\right\} \big) \cup \left\{ uw \colon w\in N_{G}(v) \right\}. 
\end{align*}
We say that $u$ and $v$ are \emph{equivalent}, denoted $u \sim v$, if $u$ and $v$ lie in the same part $V_i$ for some $i\in [3]$ and $N_G(u) = N_G(v)$ (and hence $\{u,v\} \notin G$).
For $i\in[3]$ and $v\in V_i$, write
\begin{align*}
    [v]\coloneqq \{u\in V_i \colon u \sim v\}
\end{align*}
for the equivalence class of $v$ in $V_i$ (with respect to $\sim$ in $G$).

We say that $G$ is \emph{locally symmetrized} if, for each $i\in[3]$, any two nonadjacent vertices in $V_i$ are equivalent.

Observe that the symmetrization operation extends naturally to equivalence classes, that is, we may symmetrize all vertices in $[u]$ to a vertex $[v]$ simultaneously, denoted by $G_{[u] \to [v]}$.
The advantage of this extension is that each application reduces the number of equivalence classes by one, thus, after finitely many steps, we obtain a locally symmetrized graph.
Moreover, it is easy to verify that if $G$ is cyclically triangle-free, then $G_{[u] \to [v]}$ is also cyclically triangle-free. 
In addition, if follows from the definition that for every pair of equivalence classes in $V_i$ for some $i \in [3]$, 
\begin{align*}
    |G| \le \max\left\{|G_{[u] \to [v]}|,~|G_{[v] \to [u]}|\right\}. 
\end{align*}
Therefore, in the proof of Theorem~\ref{THM:3-colored-Mantel-L1-norm-reduced}, we may assume that $G$ is locally symmetrized.

We associate to $G$ the mixed graph $\vec{G}$, obtained by orienting every edge between $V_i$ and $V_{i+1}$ from $V_i$ to $V_{i+1}$, while leaving edges inside each $V_i$ undirected.

A directed walk, path, or cycle in $\vec{G}$ is called \emph{distinct} if no two of its vertices belong to the same equivalence class in $G$. 

For $i \in [3]$ and $v \in V_i$, define its \emph{in-neighborhood}, \emph{out-neighborhood}, and \emph{internal neighborhood} by
\begin{align*}
    N^{-}(v) \coloneqq N_G(v)\cap V_{i-1}, \quad
    N^{+}(v) \coloneqq N_G(v)\cap V_{i+1}, \quad
    N^{\mathrm{int}}(v) \coloneqq N_G(v)\cap V_i.
\end{align*}


The following fact follows directly from the definition of locally symmetrized.
\begin{fact}\label{FACT:symmetrized-G}
    Suppose that $G$ is locally symmetrized. Then the following statements hold. 
    \begin{enumerate}[label=(\roman*)]
        \item\label{FACT:symmetrized-G-a} Each equivalence class is an independent set in $G$.
        \item\label{FACT:symmetrized-G-b} For each $i \in [3]$ and every pair of distinct equivalence classes $[x], [y] \subseteq V_i$, the induced bipartite subgraph $G\big[ [x], [y] \big]$ is complete. 
        In particular, for every vertex $v \in V_i$, we have $|N^{\mathrm{int}}(v)| = |V_i| - |[v]|$.
        \item\label{FACT:symmetrized-G-c} Every minimum-length directed cycle in $\vec{G}$ is distinct. 
    \end{enumerate}
\end{fact}

The following fact follows directly from the definition of cyclically triangle-free and Fact~\ref{FACT:symmetrized-G}. 
\begin{fact}\label{FACT:symmetrized-G-T3-free}
    Suppose that $G$ is locally symmetrized and cyclically triangle-free. Then the following statements hold. 
    \begin{enumerate}[label=(\roman*)]
        \item\label{FACT:symmetrized-G-T3-free-a} For every $i \in [3]$ and $v \in V_i$, either $N^{-}(v) = \emptyset$ or $N^{-}(v)  = [u]$ for some equivalence class $[u] \subseteq V_{i-1}$. 
        \item\label{FACT:symmetrized-G-T3-free-b} Let $i \in [3]$. Suppose that $(u,v) \in V_{i} \times V_{i+1}$ is a directed edge in $\vec{G}$. Then $N^{-}(u) \cap N^{+}(v) = \emptyset$. In particular, $|N^{-}(u) \cup N^{+}(v)|  \le |V_{i-1}|$.
    \end{enumerate}
\end{fact}

The following lemma is crucial for the proof of Theorem~\ref{THM:3-colored-Mantel-L1-norm-reduced}.
\begin{lemma}\label{LEMMA:deg-sum-distinct-walk}
    Suppose that $G$ is locally symmetrized and cyclically triangle-free. 
    Let $W = x_1 y_1 z_1 \ldots x_{k}y_{k}z_{k}$ be a distinct directed path in $\vec{G}$ with $(x_i, y_i, z_i) \in V_1 \times V_2 \times V_3$ for $i \in [k]$.
    Suppose that either $N^{-}_{G}(x_1) = \emptyset$ or $N^{-}_{G}(x_1) = [z_{k}]$.
    Then 
    \begin{align*}
        \sum_{i \in [k]} \big( d_{G}(x_i) + d_{G}(y_i) + d_{G}(z_i) \big)
        \le 3(k+1)n. 
    \end{align*}
\end{lemma}
\begin{proof}[Proof of Lemma~\ref{LEMMA:deg-sum-distinct-walk}]
    For $i \in [2,k]$, since $(z_{i-1}, x_i) \in \vec{G}$, it follows from Fact~\ref{FACT:symmetrized-G-T3-free}~\ref{FACT:symmetrized-G-T3-free-a} that $N^{-}(x_i) = [z_{i-1}]$. 
    Combining this with the assumption that $N^{-}_{G}(x_1) = \emptyset$ or $N^{-}_{G}(x_1) = [z_{k}]$, we obtain 
    \begin{align}\label{equ:mantel-L1-lemma-in}
        \sum_{i \in [k]} |N^{-}(x_i)|
        \le \sum_{i \in [k]} |[z_i]|. 
    \end{align}
    For $i \in [k]$, it follows from Fact~\ref{FACT:symmetrized-G}~\ref{FACT:symmetrized-G-b} that $N^{\mathrm{int}}(x_i) = |V_1| -|[x_i]| = n- |[x_i]|$. Therefore, 
    \begin{align}\label{equ:mantel-L1-lemma-int}
        \sum_{i \in [k]} |N^{\mathrm{int}}(x_i)|
        = k n - \sum_{i \in [k]} |[x_i]|. 
    \end{align}
    It follows from Fact~\ref{FACT:symmetrized-G-T3-free}~\ref{FACT:symmetrized-G-T3-free-a} that each vertex $u \in V_2$ is adjacent to at most one equivalence class in $V_1$. Thus,  
    \begin{align*}
        \sum_{i \in [k]} |N^{+}(x_i)|
        \le |V_2|
        = n. 
    \end{align*}
    Combining this with~\eqref{equ:mantel-L1-lemma-in} and~\eqref{equ:mantel-L1-lemma-int}, we obtain 
    \begin{align*}
        \sum_{i \in [k]} d_{G}(x_i)
        & = \sum_{i \in [k]} \left( |N^{+}(x_i)| + |N^{-}(x_i)| + |N^{\mathrm{int}}(x_i)| \right) \\
        & \le (k+1)n + \sum_{i \in [k]} |[z_i]| - \sum_{i \in [k]} |[x_i]|. 
    \end{align*}
    Applying the same argument to $y_1, \ldots, y_k$ and $z_1, \ldots, z_k$ yields 
    \begin{align*}
        \sum_{i \in [k]} d_{G}(y_i)
        & \le (k+1)n + \sum_{i \in [k]} |[x_i]| - \sum_{i \in [k]} |[y_i]|, \\
        \sum_{i \in [k]} d_{G}(z_i)
        & \le (k+1)n + \sum_{i \in [k]} |[y_i]| - \sum_{i \in [k]} |[z_i]|.
    \end{align*}
    Summing these three inequalities gives  
    \begin{align*}
        \sum_{i \in [k]} \big( d_{G}(x_i) + d_{G}(y_i) + d_{G}(z_i) \big)
        & \le 3 (k+1)n, 
    \end{align*}
    which proves Lemma~\ref{LEMMA:deg-sum-distinct-walk}. 
\end{proof}

We are now ready to present the proof of Theorem~\ref{THM:3-colored-Mantel-L1-norm-reduced}. 
\begin{proof}[Proof of Theorem~\ref{THM:3-colored-Mantel-L1-norm-reduced}]
    Suppose that $G$ is cyclically triangle-free and has the maximum possible number of edges. 
    By the discussion above, we may assume that $G$ is locally symmetrized. 
    Suppose to the contrary that $|G| > 5n^2/2 + 5n$. 
    
    \begin{claim}\label{CLAIM:Mantel-L1-min-triple-deg}
        Every triple of vertices $(x,y,z) \in V_1 \times V_2 \times V_3$ satisfies  
        \begin{align}\label{equ:G-L1-mindeg-assume}
            d_{G}(x) + d_{G}(y) + d_{G}(z) 
            > 5n. 
        \end{align}
    \end{claim}
    \begin{proof}[Proof of Claim~\ref{CLAIM:Mantel-L1-min-triple-deg}]
        Suppose to the contrary that $d_{G}(x) + d_{G}(y) + d_{G}(z) \le 5n$. 
        Let us label vertices in $V_1, V_2, V_3$ by $V_1 = \{x_1, \ldots, x_{n}\}$, $V_2 = \{y_1, \ldots, y_n\}$, and $V_3 = \{z_1, \ldots, z_n\}$ with $(x,y,z) = (x_1, y_1, z_1)$. 
        Since 
        \begin{align*}
            \sum_{i\in [n]} \big( d_{G}(x_{i}) + d_{G}(y_{i}) + d_{G}(z_{i}) \big)
            = 2|G|
            \ge 5n^2 + 10n, 
        \end{align*}
        by averaging, there exists a triple $(x_{i_0}, y_{i_0}, z_{i_0})$ such that 
        \begin{align*}
            d_{G}(x_{i_0}) + d_{G}(y_{i_0}) + d_{G}(z_{i_0})
            \ge 5n + 10. 
        \end{align*}
        Let $\tilde{G}$ be the graph obtained from $G$ by first removing all edges incident to $\{x_1, y_1, z_1\}$, and then symmetrizing $x_1$ to $x_{i_0}$, $y_1$ to $y_{i_0}$, and $z_1$ to $z_{1_0}$. 
        Note that $\tilde{G}$ remains cyclically triangle-free. 
        However, 
        \begin{align*}
            |\tilde{G}|
            & \ge |G| - \big( d_{G}(x) + d_{G}(y) + d_{G}(z) \big) + \big( d_{G}(x_{i_0}) -3 + d_{G}(y_{i_0})-3 + d_{G}(z_{i_0})-3 \big) \\
            & \ge |G| - 5n + (5n+10-9)
            > |G|, 
        \end{align*}
        contradicting the maximality of $G$. 
    \end{proof}

    \begin{claim}\label{CLAIM:G-L1-no-cycle}
        The mixed graph $\vec{G}$ does not contain any directed cycle. 
        Consequently, every directed walk in $\vec{G}$ is a directed path. 
    \end{claim}
    \begin{proof}[Proof of Claim~\ref{CLAIM:G-L1-no-cycle}]
        Suppose to the contrary that this claim fails.  
        Let $C = x_1 y_1 z_1 \cdots x_k y_k z_k$ be a shortest directed cycle in $\vec{G}$. By relabeling the vertices of $C$ if necessary, we may assume that $(x_i, y_i, z_i) \in V_1 \times V_2 \times V_3$ for $i \in [k]$. 
        Since $G$ does not contain triangles of type $V_1 V_2 V_3$, we have $k \ge 2$. 
        Moreover, because $C$ s chosen to be shortest, it must be distinct (otherwise, we obtain a shorter directed cycle). 

        Applying Lemma~\ref{LEMMA:deg-sum-distinct-walk} to the path $x_1 y_1 z_1 \cdots x_k y_k z_k$, we obtain 
        \begin{align*}
            \sum_{i \in [k]} \big(d_{G}(x_i) + d_{G}(y_i) + d_{G}(z_i)\big)
            \le 3(k+1)n. 
        \end{align*}
        By averaging, there exists an index $i_{\ast} \in [k]$ such that 
        \begin{align*}
            d_{G}(x_{i_{\ast}}) + d_{G}(y_{i_{\ast}}) + d_{G}(z_{i_{\ast}})
            \le 3 \left(1 + \frac{1}{k}\right)n
            \le 3 \left(1 + \frac{1}{2}\right)n
            < 5n, 
        \end{align*}
        a contradiction to~\eqref{equ:G-L1-mindeg-assume}. 
    \end{proof}

    Now fix a longest directed walk $P = w_1 \cdots w_m$ of $\vec{G}$. 
    It follows from Claim~\ref{CLAIM:G-L1-no-cycle} that $P$ is a path. 

    \begin{claim}\label{CLAIM:G-L1-longest-path}
        We have $m \ge 3$. 
    \end{claim}
    \begin{proof}[Proof of Claim~\ref{CLAIM:G-L1-longest-path}]
        Suppose to the contrary that $m \le 2$. 
        If $m = 1$, then $G[V_i, V_j] = \emptyset$ for every $\{i,j\} \in \tbinom{[3]}{2}$. In this case,
        \begin{align*}
            |G|
            = |G[V_1]| + |G[V_2]| + |G[V_3]|
            \le 3 \binom{n}{2}
            < \frac{5 n^2}{2} + 5 n,  
        \end{align*}
        a contradiction. 
        Thus, we may assume that $m = 2$. 
        
        By symmetry, we may assume that $w_1 \in V_1$ and $w_2 \in V_2$. 
        Fix an arbitrary vertex $w_3 \in V_3$. 
        It follows from the maximality of $P$ that 
        \begin{align*}
            N^{-}(w_1) = N^{+}(w_2) = \emptyset
            \quad\text{and}\quad 
            N^{+}(w_1) \cap N^{-}(w_3) = \emptyset. 
        \end{align*}
        Combining it with Fact~\ref{FACT:symmetrized-G}~\ref{FACT:symmetrized-G-b} and Fact~\ref{FACT:symmetrized-G-T3-free}~\ref{FACT:symmetrized-G-T3-free-a}, we obtain 
        %
        \begin{align*}
            \sum_{i \in [3]} d_{G}(w_i)
            & = \sum_{i \in [3]} \left(|N^{-}(w_i)| + |N^{\mathrm{int}}(w_i)| + |N^{+}(w_i)|\right) \\
            & \le \big( n - |[w_1]| + |N^{+}(w_1)| \big) + \big( |[w_1]| + n - |[w_2]| \big) + \big( |N^{-}(w_3)| + 2n \big) \\[0.5em]
            & \le 4n + |N^{+}(w_1)| + |N^{-}(w_3)| - |[w_2]|
            \le 5n - |[w_2]|
            < 5n, 
        \end{align*}
        a contradiction to~\eqref{equ:G-L1-mindeg-assume}. 
    \end{proof}

    We now consider the last three vertices $\{w_{m-2}, w_{m-1}, w_{m}\}$ of the path $P$. 
    By symmetry, we may assume that $(w_{m-2}, w_{m-1}, w_{m}) \in V_1 \times V_2 \times V_3$. 
    For convenience, set $(x,y,z) \coloneqq (w_{m-2}, w_{m-1}, w_{m})$. 
    Since $P$ is a longest directed path, we have $N^{+}(z) = \emptyset$. 
    Moreover, as $(x,y)$ is a directed edge, it follows from Fact~\ref{FACT:symmetrized-G-T3-free}~\ref{FACT:symmetrized-G-T3-free-b} that $|N^{-}(x)| + |N^{+}(y)| \le |V_3| = n$. 
    Therefore, 
    \begin{align*}
        &  d_{G}(x) + d_{G}(y) + d_{G}(z) \\[0.3em]
        & \qquad \le \big( |N^{-}(x)| + n - |[x]| + n \big)   + \big( |[x]| + n - |[y]| + |N^{+}(y)| \big) + \big( |[y]| + n \big) \\[0.3em]
        & \qquad = 4n + |N^{-}(x)| + |N^{+}(y)|
        \le 5n, 
    \end{align*}
    a contradiction to~\eqref{equ:G-L1-mindeg-assume}. 
    This completes the proof of Theorem~\ref{THM:3-colored-Mantel-L1-norm}. 
\end{proof}

\subsection{Proof of Theorem~\ref{THM:3-colored-Mantel-stability}}\label{SUBSEC:proof-3-colored-Mantel-stability}
In this subsection, we present the proof of Theorem~\ref{THM:3-colored-Mantel-stability}. We will use the classical result of Bollob{\'a}s--Erd{\H o}s--Straus~\cite{BES74}, together with a stability version recently established in~\cite{CLY25}. For our purposes, we state slightly modified variants of these results, which follow straightforwardly from the Removal Lemma.

\begin{theorem}[\cite{BES74} and {\cite[Theorem~1.5]{CLY25}}]\label{THM:3-partite-Mantel}
    For every $\xi > 0$ there exist $\delta_{\ref{THM:3-partite-Mantel}} =\delta_{\ref{THM:3-partite-Mantel}}(\xi) > 0$ and $N_{\ref{THM:3-partite-Mantel}} = N_{\ref{THM:3-partite-Mantel}}(\xi)$ such that the following holds for every $n \ge N_{\ref{THM:3-partite-Mantel}}$. 
    Let $U_1, U_2, U_3$ be three pairwise disjoint sets and $H$ be a $3$-partite graph with parts $U_1$, $U_2$, and $U_3$.
    Suppose that $\max_{i \in [3]}\left| |U_i| - n \right| \le \delta_{\ref{THM:3-partite-Mantel}} n$ and the number of $K_{3}$ in $H$ is at most $\delta_{\ref{THM:3-partite-Mantel}} n^3$. 
    Then 
    \begin{align*}
        |H| \le 2n^2 + \xi n^2. 
    \end{align*}
    Moreover, if $|H| \ge 2n^2 - \delta_{\ref{THM:3-partite-Mantel}} n^2$, then there exist $i \in [3]$ and a partition $U_{i1} \cup U_{i2} = U_i$ such that 
    \begin{align*}
        | H \triangle K[U_{i1} \cup U_{i+1}, U_{i2} \cup U_{i+2}] |
        \le \xi n^2. 
    \end{align*}
\end{theorem}

\begin{proof}[Proof of Theorem~\ref{THM:3-colored-Mantel-stability}]
    Fix $\varepsilon > 0$. 
    Let $\delta, \delta_{1}, \delta_{2} > 0$ be sufficiently small such that $\delta \ll \delta_{1} \ll \delta_{2} \ll \varepsilon$. 
    Let $n$ be sufficiently large.
    Let $V_1, V_2, V_3$ be pairwise disjoin sets satisfying $\max_{i \in [3]}\left| |V_i| - n \right| \le \delta n$. 
    Suppose that $G$ satisfies $\rho_{3}(G) \le \delta$ and 
    \begin{align}\label{equ:THM:3-colored-Mantel-stability}
        |G| + |G[V_1, V_2, V_3]| 
        \ge \frac{9 n^2}{2} - \delta n^2. 
    \end{align}
    We can choose $\delta$ sufficiently small such that, by Theorems~\ref{THM:3-colored-Mantel-L1-norm} and~\ref{THM:3-partite-Mantel}, 
    \begin{align}\label{equ:Mantel-stab-3upper-bounds}
        \max\left\{ |G| - \frac{5n^2}{2},~|G[V_1, V_2, V_3]| - 2n^2 \right\}
        \le \delta_{1}n^2,
    \end{align}
    and moreover, such that, by the Removal Lemma, there exists a subgraph $\tilde{G} \subseteq G$ with $\rho_{3}(\tilde{G}) = 0$ and $|\tilde{G}| \ge |G| - \delta_1 n^2$. 
    Combining this with~\eqref{equ:THM:3-colored-Mantel-stability} and~\eqref{equ:Mantel-stab-3upper-bounds}, we obtain 
    \begin{align}\label{equ:Mantel-stab-3lower-bounds}
        \min\left\{ |\tilde{G}| - \frac{5n^2}{2},~|\tilde{G}[V_1, V_2, V_3]| - 2n^2 \right\}
        \ge -3\delta_{1}n^2.
    \end{align}
    Applying Theorem~\ref{THM:3-partite-Mantel} to the $3$-partite graph $\tilde{G}[V_1, V_2, V_3]$, we conclude that there exist $i \in [3]$ and a partition $V_{i1} \cup V_{i2} = V_i$ such that 
    \begin{align}\label{equ:Mantel-stab-sym-diff}
        |\tilde{G}[V_1, V_2, V_3] \triangle K[V_{i1} \cup V_{i+1}, V_{i2} \cup V_{i+2}]|
        \le \delta_{2} n^2. 
    \end{align}
    By symmetry, we may assume that $i = 1$.

    \begin{claim}\label{CLAIM:K43-L1-stab-inside-edges}
        The following statements hold. 
        \begin{enumerate}[label=(\roman*)]
            \item\label{CLAIM:K43-L1-stab-inside-edges-1} $\max\left\{|\tilde{G}[V_{11}, V_{2}]|,~|\tilde{G}[V_{12}, V_{3}]|,~|\tilde{G}[V_{2}]|,~|\tilde{G}[V_{12}]|\right\} \le \delta_{2} n^2$. 
            \item\label{CLAIM:K43-L1-stab-inside-edges-2} $|V_{11}| \le \delta_{2}^{1/2} n$ or $|\tilde{G}[V_{3}]| \le \delta_{2}^{1/2} n^2$. 
        \end{enumerate}
    \end{claim}
    \begin{proof}[Proof of Claim~\ref{CLAIM:K43-L1-stab-inside-edges}]
        Let $M \coloneqq K[V_{11} \cup V_2, V_{12} \cup V_3] \setminus \tilde{G}[V_1, V_2, V_3]$. 
        It follows from~\eqref{equ:Mantel-stab-sym-diff} that $|M| \le \delta_{2} n^2$. 
        First, we prove that $|\tilde{G}[V_{11}, V_{2}]| \le \delta_{2} n^2$. 
        Let 
        \begin{align*}
            \mathcal{T}
            \coloneqq \big\{\{u,v,w\} \colon \text{$uv \in \tilde{G}[V_{11}, V_{2}]$, $w \in V_3$, and $\{uw, vw\} \cap M \neq \emptyset$} \big\}. 
        \end{align*}
        Fix $uv \in \tilde{G}[V_{11}, V_{2}]$. 
        Since $\tilde{G}$ does not contain triangles of type $V_1V_2V_3$, for every $w\in V_3$ we have $\{uw, vw\} \cap M \neq \emptyset$. 
        It follows that $|\mathcal{T}| \ge |\tilde{G}[V_{11}, V_{2}]| |V_{3}|$. 
        On the other hand, every pair $uw \in M$ is contained in at most $n$ members of $\mathcal{T}$. 
        Thus, we have $|\mathcal{T}| \le n|M|$. 
        Combining the two bounds, we obtain 
        \begin{align*}
            |\tilde{G}[V_{11}, V_{2}]| 
            \le \frac{n |M|}{|V_{3}|}
            = |M|
            \le \delta_{2} n^2. 
        \end{align*}
        The proofs for the bounds on $|\tilde{G}[V_{12}, V_{3}]|$, $|\tilde{G}[V_2]|$ (here we use the assumption that $\tilde{G}$ contains no triangles of type $V_2V_2V_3$), $|\tilde{G}[V_{12}]|$ (here we use the assumption that $\tilde{G}$ contains no triangles of type $V_1V_1V_2$), and for Claim~\ref{CLAIM:K43-L1-stab-inside-edges}~\ref{CLAIM:K43-L1-stab-inside-edges-2} (here we use the assumption that $\tilde{G}$ contains no triangles of type $V_3V_3V_1$) are similar, thus we omit them here. 
    \end{proof}

    %
    Suppose that $|V_{11}| \le \delta_{2}^{1/2} n$. 
    Then it follows from Claim~\ref{CLAIM:K43-L1-stab-inside-edges}~\ref{CLAIM:K43-L1-stab-inside-edges-1} that 
    \begin{align*}
        & |\tilde{G}[V_1]| + |\tilde{G}[V_2]| + |\tilde{G}[V_1, V_{3}]| \\
        & \qquad = |\tilde{G}[V_{11}]| + |\tilde{G}[V_{11}, V_{12}]| + |\tilde{G}[V_{12}]| + |\tilde{G}[V_2]| + |\tilde{G}[V_{11}, V_{3}]|  + |\tilde{G}[V_{12}, V_{3}]| \\
        & \qquad \le \tbinom{|V_{11}|}{2} + |V_{11}| n + \delta_{2} n^2 + \delta_{2} n^2 + |V_{11}|n + \delta_{2} n^2 
        \le 4\delta_2 n^2 + 2\delta_2^{1/2} n^2
        \le \varepsilon n^2, 
    \end{align*}
    as desired. 

    Suppose that $|V_{11}| > \delta_{2}^{1/2} n$.
    Then it follows from Claim~\ref{CLAIM:K43-L1-stab-inside-edges}~\ref{CLAIM:K43-L1-stab-inside-edges-2} that $|\tilde{G}[V_{3}]| \le \delta_{2}^{1/2} n^2$. 
    Combining this with Claim~\ref{CLAIM:K43-L1-stab-inside-edges}~\ref{CLAIM:K43-L1-stab-inside-edges-1}, we obtain  
    \begin{align}\label{equ:Mantel-stab-V1V2V2}
        & |\tilde{G}[V_1]| + |\tilde{G}[V_2]| + |\tilde{G}[V_{3}]| \notag \\
        & \qquad = |\tilde{G}[V_{11}]| + |\tilde{G}[V_{11}, V_{12}]| + \tilde{G}[V_{12}] + |\tilde{G}[V_2]| + |\tilde{G}[V_{3}]| \notag \\
        & \qquad \le \tbinom{|V_1|}{2} - \tbinom{|V_{12}|}{2} + \delta_{2} n^2 + \delta_{2} n^2 + \delta_{2} n^2
        \le \frac{n^2}{2}   + 3\delta_{2}^{1/2} n^2 - \frac{|V_{12}|^2}{2}. 
    \end{align}
    On the other hand, it follows from~\eqref{equ:Mantel-stab-3upper-bounds} and~\eqref{equ:Mantel-stab-3lower-bounds} 
    \begin{align*}
        |\tilde{G}[V_1]| + |\tilde{G}[V_2]| + |\tilde{G}[V_{3}]|
        & = |\tilde{G}| - |\tilde{G}[V_1, V_2, V_3]| \\
        & \ge \frac{5}{2} n^2 - 3\delta_{1} n^2 - \left(2n^2+\delta_{1}n^2)\right)
        \ge \frac{1}{2} n^2 - 4\delta_{1} n^2. 
    \end{align*}
    Combining this with~\eqref{equ:Mantel-stab-V1V2V2}, we obtain 
    \begin{align*}
        \frac{|V_{12}|^2}{2} 
        \le 3\delta_{2}^{1/2} n^2 + 4\delta_{1} n^2 
        \le 4\delta_{2}^{1/2} n^2, 
    \end{align*}
    which implies that $|V_{12}| \le 2\sqrt{2}\delta_{2}^{1/4} n$.
    Consequently, 
    \begin{align*}
        & |\tilde{G}[V_{2}]| + |\tilde{G}[V_{3}]| + |\tilde{G}[V_1, V_2]| \\
        & \qquad = |\tilde{G}[V_{2}]| + |\tilde{G}[V_{3}]| + |\tilde{G}[V_{11}, V_2]| + |\tilde{G}[V_{12}, V_2]| \\
        & \qquad = \delta_{2} n^2 + \delta_{2}^{1/2} n^2 + \delta_{2} n^2 +  2\sqrt{2}\delta_{2}^{1/4} n^2 
        \le \varepsilon n^2, 
    \end{align*}
    as desired.     
    This completes the proof of Theorem~\ref{THM:3-colored-Mantel-stability}.  
\end{proof}

\subsection{Proof of Theorem~\ref{THM:3-colored-Mantel-L2-norm}}\label{SUBSEC:proof-3-colored-Mantel=L2}
In this section, we present the proof of Theorem~\ref{THM:3-colored-Mantel-L2-norm}, derived by computer using the flag algebra method of Razborov~\cite{Raz07}, which is also described in e.g.~\cite{Razborov10,BT11,SFS16,GilboaGlebovHefetzLinialMorgenstein22}.
Since this method is well-known by now, we will be very brief. In particular,  we omit many definitions, referring
the reader to~\cite{Raz07,Razborov10} for any missing notions. Roughly speaking, a \emph{flag algebra proof using $0$-flags on $m$ vertices} of an upper bound $u\in\mathbb{R}$ on the given objective function $f$ consists of an identity 
\begin{align*}
    u-f(\mathcal{H})
    = \mathrm{SOS}+\sum_{F\in\mathcal{F}_m^0}c_F \cdot p(F,\mathcal{H})+o(1),
\end{align*}
which is asymptotically true for any admissible $\mathcal{H}$ with $|V(\mathcal{H})|\to\infty$, where the $\mathrm{SOS}$-term can be represented as a sum of squares (as described e.g. in~{\cite[Section~3]{Razborov10}}), each coefficient $c_F\in\mathbb{R}$ is non-negative, and $\mathcal{F}_m^0$ consists of isomorphism types of \emph{$0$-flags} (unlabeled $3$-graphs) with $m$ vertices. If $f(\mathcal{H})$ can be represented as a linear combination of the densities of members of $\mathcal{F}_m^0$ in $\mathcal{H}$ then finding the smallest possible $u$ amounts to solving a semi-definite program (SDP) with $|\mathcal{F}_m^0|$ linear constraints (so we write the size of $\mathcal{F}_m^0$ in each case to give the reader some idea of the size of the programs that we had to solve).

We formed the corresponding SDPs and then analyzed the solutions returned by computer, using a modified version of the SageMath package, which can be found in the GitHub repository \href{https://github.com/bodnalev/sage}{\url{https://github.com/bodnalev/sage}}. 
The script together with the certificate can be found at \href{https://github.com/xliu2022/xliu2022.github.io/tree/main/FlagAlgebra_Certificate/ThreeColored-Mantel-L2norm}{\url{https://github.com/xliu2022/xliu2022.github.io/tree/main/FlagAlgebra_Certificate/ThreeColored-Mantel-L2norm}}.

\begin{proof}[Proof of Theorem~\ref{THM:3-colored-Mantel-L2-norm}]
    Suppose to the contrary that this theorem fails. 
    Then there exists a constant $\varepsilon > 0$, a decreasing infinite sequence $\left(\delta_i \right)_{i = 1}^{\infty}$ of positive real numbers with $\lim_{i \to \infty} \delta_i = 0$, and an increasing (i.e., the number of vertices is strictly increasing) infinite sequence  $\left(G_{i}\right)_{i=1}^{\infty}$ of graphs such that 
    \begin{enumerate}[label=(\roman*)]
        \item\label{equ:K43-flag-Vi-size} $G_i$ is a graph on $V^{i} = V_{1}^{i} \cup V_{2}^{i} \cup V_{3}^{i}$, where $V_{1}^{i}, V_{2}^{i}, V_{3}^{i}$ are pairwise disjoint and satisfy  
        \begin{align*}
            \left| V_{j}^{i} - \tfrac{|V^{i}|}{3} \right|
            \le \delta_{i} |V^{i}|
            \quad\text{for}\quad j \in [3].
        \end{align*}
        \item\label{equ:K43-flag-local-max} $G_i$ is locally maximal with respect to $(V_{1}^{i}, V_{2}^{i}, V_{3}^{i})$, that is, up to a cyclic permutation of the sub-indices,
        \begin{align*}
            |G_{i}[V_1^{i}, V_2^{i}]| + |G_{i}[V_3^{i}]| & \ge |G_{i}[V_1^{i}, V_3^{i}]| + |G_{i}[V_1^{i}]|,  \\ 
            |G_{i}[V_2^{i}, V_3^{i}]| + |G_{i}[V_3^{i}]| & \ge |G_{i}[V_1^{i}, V_3^{i}]| + |G_{i}[V_2^{i}]|. 
        \end{align*}
        \item\label{equ:K43-flag-K3-density} $\rho_{3}(G_i)  \le \delta_{i}$.
        \item\label{equ:K43-flag-L2norm-density}  $\norm{G_i}_{2}  \ge (1/3 + \varepsilon) |V^{i}|^{3}$.
    \end{enumerate}

    We would like to run flag algebra calculations on the limit $\phi$ of the sequence $\left(G_{i}\right)_{i=1}^{\infty}$ in the theory of $3$-colored graphs, that is, we have three unary relations $U_1, U_2, U_3$ with the only restriction that each vertex in a flag satisfies exactly one of them, equivalently, these unitary relations encode a $3$-partition of the vertex set.

    For $i \in [3]$, let $v_i$ denote the untyped single-vertex flag whose vertex has color $i$. 
    Then, in the limit, part~\ref{equ:K43-flag-Vi-size} translates to the statement that
    \begin{align}\label{equ:K43-flag-Vi-size-phi}
        \phi(v_i) = 1/3
        \quad\text{for}\quad i \in [3]. 
    \end{align}
    For $(i,j) \in [3]^{2}$, let $E_{i,j}$ denote the untyped two-vertex flag with one vertex colored $i$ and the other vertex colored $j$. 
    Then, in the limit, part~\ref{equ:K43-flag-local-max} translates to the statement that
    \begin{equation}
    \begin{aligned}
    \label{equ:K43-flag-local-max-phi}
        \phi(E_{1,2}) + \phi(E_{3,3}) - \phi(E_{1,3}) - \phi(E_{1,1}) & \ge 0, \\
        \phi(E_{2,3}) + \phi(E_{3,3}) - \phi(E_{1,3}) - \phi(E_{2,3}) & \ge 0.
    \end{aligned}
    \end{equation}
    For $(i,j,k) \in [3]^{3}$, let $T_{i,j,k}$ denote the untyped three-vertex flag with one vertex colored $i$, one vertex colored $j$, and the remaining vertex colored $k$. 
    Then, in the limit, part~\ref{equ:K43-flag-K3-density} translates to the statement that
    \begin{align}\label{equ:K43-flag-K3-density-psi}
        \phi(T_{1,2,3}) + \phi(T_{1,1,2}) + \phi(T_{2,2,3}) + \phi(T_{3,3,1})
        = 0. 
    \end{align}
    Thus, we can restrict the flag algebra calculations to the theory of $\{T_{1,2,3},T_{1,1,2},T_{2,2,3},T_{3,3,1}\}$-free $3$-colored graphs.
    
    Let $S_{2}$ denote the two-edge star and $K_{3}$ denote the triangle. 
    For a graph $F$, let $\rho_{\mathrm{ind}}(F,G_i)$ denote the density of induced copies of $F$ in $G_i$, i.e., the probability that the induced subgraph of $G$ on a uniformly sampled $v(F)$-subset of $V(G)$ is isomorphic to $F$.
    By definition,
    \begin{align*}
        \norm{G_i}_{2}
        = \sum_{v\in V^{i}}d_{G_i}^{2}(v)
        & = \sum_{v\in V^{i}} \left( 2\tbinom{d_{G_i}(v)}{2} + d_{G_i}(v) \right) \\ 
        & = 2\big( \rho_{\mathrm{ind}}(S_{2},G_i) + 3 \rho_{\mathrm{ind}}(K_3,G_i) \big) \tbinom{|V^{i}|}{3} + 2|G_i|. 
    \end{align*} 
    Let $\mathcal{S}$ denote the collection of untyped three-vertex flag with exactly two edges and arbitrary vertex colors (there are $18$ such flags).
    Let $\mathcal{K}$ denote the collection of untyped three-vertex flag with three edges and arbitrary vertex colors. 
    Then, by the equality above, part~\ref{equ:K43-flag-L2norm-density} translates, in the limit, to the statement that
    \begin{align}\label{equ:K43-flag-L2norm-density-psi}
        \sum_{S\in \mathcal{S}}\phi(S) + 3 \sum_{K \in \mathcal{K}}\phi(K)
        \ge 1 + 3\varepsilon. 
    \end{align}

    We can now run the usual flag algebra calculations in the theory of $\{T_{1,2,3},T_{1,1,2},T_{2,2,3},T_{3,3,1}\}$-free $3$-colored graphs, where each of the inequalities in~\eqref{equ:K43-flag-Vi-size-phi} to~\eqref{equ:K43-flag-local-max-phi} can be multiplied by an unknown non-negative combination of respectively $0$-flags. 
    The final inequality should prove~\eqref{equ:K43-flag-L2norm-density-psi}.

    Our computation uses $5$-vertex flags (where $|\mathcal{F}_5^0|=1968$). 
    The output from the computer shows that 
    \begin{align*}
        \sum_{S\in \mathcal{S}}\phi(S) + 3 \sum_{K \in \mathcal{K}}\phi(K)
        \le 1, 
    \end{align*}
    which contradicts~\eqref{equ:K43-flag-L2norm-density-psi}. 
    This completes the proof of Theorem~\ref{THM:3-colored-Mantel-L2-norm}.
\end{proof}

\section{Proof of Theorem~\ref{THM:L2-exact-K43}}\label{SEC:proof-THM-L2-exact-K43}
%
\subsection{Preliminaries}\label{SEC:prelim}
Given pairwise disjoint sets $V_1, \ldots, V_{t}$, denote by $K[V_1, \ldots, V_{k}]$ the complete $k$-partite graph with parts $V_1, \ldots, V_{t}$. 
Denote by $\overline{K}[V_1, \ldots, V_{k}]$ the complement of $K[V_1, \ldots, V_{k}]$, that is, 
\begin{align*}
    \overline{K}[V_1, \ldots, V_{k}]
    \coloneqq \tbinom{V_1}{2} \cup \cdots \cup \tbinom{V_{k}}{2}. 
\end{align*}

Let $\mathcal{H}$ be a $3$-graph on vertex set $V$.
The \emph{shadow} of $\mathcal{H}$ is defined as 
\begin{align*}
    \partial\mathcal{H}
    \coloneqq \left\{e \in \tbinom{V}{2} \colon \text{there exists $E \in \mathcal{H}$ such that $e \subseteq E$} \right\}. 
\end{align*}
The \emph{link} of a vertex $v \in V$ in $\mathcal{H}$ is given by 
\begin{align*}
    L_{\mathcal{H}}(v)
    \coloneqq \big\{ e \in \partial\mathcal{H} \colon \{v\} \cup e \in \mathcal{H} \big\}. 
\end{align*}
The \emph{degree} of $v$ is $d_{\mathcal{H}}(v) \coloneqq |L_{\mathcal{H}}(v)|$. 
We use $\delta(\mathcal{H})$, $d(\mathcal{H})$, and $\Delta(\mathcal{H})$ to denote the \emph{minimum}, \emph{average}, and \emph{maximum} degree of $\mathcal{H}$, respectively.

The \emph{neighborhood} of $v$ in $\mathcal{H}$ is defined as 
\begin{align*}
    N_{\mathcal{H}}(v)
    \coloneqq \big\{u \in V \setminus \{v\} \colon \text{there exists $E \in \mathcal{H}$ such that $\{u,v\} \subseteq E$} \big\}.
\end{align*}
It follows from the definition that $u \in N_{\mathcal{H}}(v)$ if and only if $u \neq v$ and $d_{\mathcal{H}}(uv) \neq 0$. 
Moreover, for every $u \in V\setminus \{v\}$, the degree of $u$ in the link graph $L_{\mathcal{H}}(v)$ is the same as the codegree $d_{\mathcal{H}}(uv)$. 

For a $2$-subset $e \subseteq V$, the \emph{neighborhood} of $e$ in $\mathcal{H}$ is defined as 
\begin{align*}
    N_{\mathcal{H}}(e) 
    \coloneqq \big\{ v\in V \colon e \cup \{v\} \in \mathcal{H} \big\}.  
\end{align*}
Recall that the \emph{codegree} of $e$ in $\mathcal{H}$ is the same as the size of $N_{\mathcal{H}}(e)$.

For a vertex subset $S\subseteq V(\mathcal{H})$, we use $\mathcal{H}[S]$ to denote the \emph{induced subgraph} of $\mathcal{H}$ on $S$, and use $\mathcal{H} - S$ to denote the induced subgraph of $\mathcal{H}$ on $V(\mathcal{H})\setminus S$. 
If $S = \{v\}$ consists of a single vertex, then we write $\mathcal{H} - v$ instead of $\mathcal{H} - \{v\}$ for simplicity.

Following the definition in~\cite{CILLP24}, the \emph{$2$-norm degree} of a vertex $v\in V$ is defined as 
\begin{align*}
    s_{\mathcal{H}}(v)
    \coloneqq \norm{\mathcal{H}}_{2} - \norm{\mathcal{H} - v}_{2}.
\end{align*}
Straightforward calculations (see~{\cite[Lemma~3.1]{CILLP24}}) show that 
\begin{align}\label{equ:def-2norm-degree-b}
    s_{\mathcal{H}}(v)
    & = \norm{L_{\mathcal{H}}(v)}_{2} + 2 \sum_{e\in L_{\mathcal{H}}(v)} d_{\mathcal{H}}(e) - d_{\mathcal{H}}(v). 
\end{align}
Denote by $s(\mathcal{H})$ the \emph{average $2$-norm degree} of $\mathcal{H}$. 

For convenience, we will omit the subscript $\mathcal{H}$ from the above notations when it is clear from the context. 

For the remainder of this section, all indices are taken modulo $3$. 

The following statements can be verified by straightforward calculations (a proof of Fact~\ref{FACT:Cn-L2-max} is provided in the appendix of the arXiv version). 

\begin{fact}\label{FACT:Cn-L2-max}
    Let $n \ge 6$ be an integer, and let $V_1 \cup V_2 \cup V_3 = [n]$ be a partition of $[n]$. 
    Then 
    \begin{align*}
        \norm{\mathbb{C}[V_1, V_2, V_3]}_{2} \le \norm{\mathbb{C}_{n}}_{2}. 
    \end{align*}
\end{fact}

\begin{fact}\label{FACT:C123-2norm}
    Let $V_1, V_2, V_3$ be three pairwise disjoint sets with $|V_i| = x_i n$ for $i \in [3]$. 
    Then 
    \begin{multline}\label{equ:C123-2norm-exp}
        \norm{\mathbb{C}[V_1, V_2, V_3]}_{2} \\
        = \sum_{i\in [3]}\left(\tfrac{x_{i} x_{i+1} \left(2 \left(x_{i}+x_{i+2}\right)^2+x_{i} x_{i+1}\right)}{2} n^4 -\tfrac{x_{i} x_{i+1} \left(4 x_{i}+x_{i+1}+4 x_{i+2}\right)}{2} n^3 +  x_{i} x_{i+1} n^2\right).
    \end{multline}
    In particular, let $\delta \in (0,1/3)$ and suppose $n \ge \frac{9}{2\delta^2}$.  If $x_i \ge 1/3 - \delta$ for $i \in [3]$. Then 
    \begin{align*}
        \norm{\mathbb{C}[V_1, V_2, V_3]}_{2}
        \ge \frac{n^4}{6} - 2\delta n^4. 
    \end{align*}
\end{fact}
\begin{proof}[Proof of Fact~\ref{FACT:C123-2norm}]
    Equality~\eqref{equ:C123-2norm-exp} follows directly from the definition of $\mathbb{C}[V_1, V_2, V_3]$ and the definition of $2$-norm. 
    So it suffices to prove the second part of the statement. 
    Suppose that $x_i \ge 1/3 -\delta$ for $i \in [3]$. Then 
    \begin{align*}
        \frac{x_{i} x_{i+1} \left(2 \left(x_{i}+x_{i+2}\right)^2+x_{i} x_{i+1}\right)}{2} n^4
        & \ge \frac{n^4}{18} - \frac{2\delta n^4}{3} + \left(3\delta^2 - 6\delta^3 + \frac{9\delta^4}{2}\right) n^4 \\
        & \ge \frac{n^4}{18} - \frac{2\delta n^4}{3} + \delta^2 n^4, 
    \end{align*}
    where the last inequality uses the fact that $2\delta^2 - 6\delta^3 + \frac{9\delta^4}{2} \ge 0$ for $\delta \in [0, 1/3]$. 

    On the other hand, plugging in $x_1 = x_2 = x_3 = 1$, we obtain the trivial upper bound 
    \begin{align*}
        \frac{x_{i} x_{i+1} \left(4 x_{i}+x_{i+1}+4 x_{i+2}\right)}{2} n^3
        \le \frac{9 n^3}{2}. 
    \end{align*}
    Combining this with the inequality above, we obtain 
    \begin{align*}
        \norm{\mathbb{C}[V_1, V_2, V_3]}_{2}
        \ge 3\left(\frac{n^4}{18} - \frac{2\delta n^4}{3} + \delta^2 n^4 - \frac{9 n^3}{2} \right) 
        \ge \frac{n^4}{6} - 2\delta n^4, 
    \end{align*}
    completing the proof of Fact~\ref{FACT:C123-2norm}. 
\end{proof}

We will use the following inequality, which can be established by standard calculus, for convenience, we verified it using Mathematica\footnote[2]{A Mathematica notebook containing the calculations is available at \href{https://github.com/xliu2022/xliu2022.github.io/blob/main/Mathematica_Calculations/TetrahedronL2Norm.nb}{\url{https://github.com/xliu2022/xliu2022.github.io/blob/main/Mathematica_Calculations/TetrahedronL2Norm.nb}}. }.
\begin{lemma}\label{LEMMA:3-part-inequality}
    Suppose that $(x_1, x_2, x_3) \in \mathbb{R}^{3}$ are nonnegative real numbers satisfying $x_1 + x_2 + x_3 = 1$. Then 
    \begin{align*}
        x_1 x_2 x_3 + \frac{x_1^2 x_2}{2} + \frac{x_2^2 x_3}{2} + \frac{x_3^2 x_1}{2} 
        \le \frac{5}{54} - \frac{1}{50} \left( \left(x_1 - \tfrac{1}{3}\right)^2 + \left(x_2 - \tfrac{1}{3}\right)^2 + \left(x_3 - \tfrac{1}{3}\right)^2 \right). 
    \end{align*}
\end{lemma}

Since the shadow of $K_{4}^{3}$ is a complete graph, duplicating a vertex in a $K_{4}^{3}$-free $3$-graph does not create any copy of $K_{4}^{3}$. 
Thus, a simple deletion–duplication argument (see e.g.~{\cite[Lemma~2.6]{HLZ25Fano}}) shows that every extremal (in terms of $\ell_{2}$-degree) $K_{4}^{3}$-free $3$-graph is early regular in the $\ell_{2}$-norm degree. 

\begin{lemma}\label{LEMMA:2norm-degree-regular}
    Suppose that $\mathcal{H}$ is an $K_{4}^{3}$-free $n$-vertex $3$-graph with $\norm{\mathcal{H}}_{2} = \mathrm{ex}_{\ell_2}(n, K_{4}^{3})$.
    Then for every pair of vertices $\{u,v\} \subseteq V(\mathcal{H})$, we have $\left|s_{\mathcal{H}}(u) - s_{\mathcal{H}}(v)\right| \le 60 n^2$. 
    In particular, for every vertex $v \in V(\mathcal{H})$, 
    \begin{align*}
        \left|s_{\mathcal{H}}(v) - s(\mathcal{H})\right| 
        \le 60 n^2.
    \end{align*}
\end{lemma}

The following stability theorem of Balogh--Clemen--Lidick\'{y} will be crucial for our proof. 

\begin{theorem}[{\cite[Theorem~1.5]{BCL22b}}]\label{THM:K43-L2-stability}
    For every $\varepsilon >0$, there exist $\delta_{\ref{THM:K43-L2-stability}}= \delta_{\ref{THM:K43-L2-stability}}(\varepsilon) > 0$ and $N_{\ref{THM:K43-L2-stability}} = N_{\ref{THM:K43-L2-stability}}(\varepsilon)$ such that the following holds for every $n \ge N_{\ref{THM:K43-L2-stability}}$.
    Suppose that $\mathcal{H}$ is an $n$-vertex $K_4^3$-free $3$-graph with $\norm{\mathcal{H}}_2 \ge \left(\frac{1}{6}- \delta_{\ref{THM:K43-L2-stability}} \right) n^4$. 
    Then 
    \begin{align*}
        |\mathcal{H} \triangle \mathbb{C}_n| 
        \le \varepsilon n^3. 
    \end{align*}
\end{theorem}

\subsection{Two key lemmas}\label{SUBSEC:two-key-lemma}
In this subsection, we establish two lemmas providing local modifications (see Figure~\ref{fig:local-adjust}) of $3$-graphs ``close'' to $\mathbb{C}_{n}$ that increase their $\ell_{2}$-norm.
In these lemmas, we do not require the $K_{4}^{3}$-freeness. 

Let $\mathcal{H}$ be a $3$-graph. 
For every pair $e \subseteq V(\mathcal{H})$, define 
\begin{align*}
    \mathcal{H}(e)
    \coloneqq \left\{ E\in \mathcal{H} \colon e\subseteq E \right\}.
\end{align*}
Given a partition $V_1 \cup V_2 \cup V_{3} = V(\mathcal{H})$, let 
\begin{align}\label{equ:def-bad-missing-triples}
    \mathcal{B}_{\mathcal{H}}[V_1, V_2, V_3]
    \coloneqq \mathcal{\mathcal{H}} \setminus \mathbb{C}[V_1, V_2, V_3]
    \quad\text{and}\quad 
    \mathcal{M}_{\mathcal{H}}[V_1, V_2, V_3]
    \coloneqq \mathbb{C}[V_1, V_2, V_3] \setminus \mathcal{H}. 
\end{align}
We refer to $\mathcal{B}_{\mathcal{H}}[V_1, V_2, V_3]$ and $\mathcal{M}_{\mathcal{H}}[V_1, V_2, V_3]$ as the sets of \emph{bad edges} and \emph{missing edges} of $\mathcal{H}$ (with respects to $(V_1, V_2, V_3)$). 

We further partition the set of bad edges by setting 
\begin{equation}
\begin{aligned}
    \mathcal{B}^{\mathrm{int}}_{\mathcal{H}}[V_1, V_2, V_3]
    & \coloneqq \left\{e \in \mathcal{B}_{\mathcal{H}}[V_1, V_2, V_3] \colon \text{$e\subseteq V_i$ for some $i \in [3]$}\right\}, \\
    \mathcal{B}^{\mathrm{bi}}_{\mathcal{H}}[V_1, V_2, V_3]
    & \coloneqq \mathcal{B}_{\mathcal{H}}[V_1, V_2, V_3] \setminus \mathcal{B}^{\mathrm{int}}_{\mathcal{H}}[V_1, V_2, V_3]. 
\end{aligned}
\label{equ:def-B-int-bi}
\end{equation}
Partition the set of missing edges by setting 
\begin{equation}
\begin{aligned}
    \mathcal{M}^{\mathrm{tri}}_{\mathcal{H}}[V_1, V_2, V_3]
    & \coloneqq \left\{e \in \mathcal{M}_{\mathcal{H}}[V_1, V_2, V_3] \colon \text{$|e \cap V_i| = 1$ for $i \in [3]$}\right\}, \\
    \mathcal{M}^{\mathrm{bi}}_{\mathcal{H}}[V_1, V_2, V_3]
    & \coloneqq \mathcal{M}_{\mathcal{H}}[V_1, V_2, V_3] \setminus \mathcal{M}^{\mathrm{tri}}_{\mathcal{H}}[V_1, V_2, V_3]. 
\end{aligned}
\label{equ:def-M-tri-bi}
\end{equation}
We will omit the subscript $\mathcal{H}$ and the bracket $[V_1, V_2, V_3]$ if it is clear from the context. 


\begin{figure}[H]
\centering

\tikzset{every picture/.style={line width=1pt}} 

\begin{tikzpicture}[x=0.75pt,y=0.75pt,yscale=-1,xscale=1,line join=round, scale=0.8]

\draw   (163,87.5) .. controls (163,65.13) and (181.13,47) .. (203.5,47) .. controls (225.87,47) and (244,65.13) .. (244,87.5) .. controls (244,109.87) and (225.87,128) .. (203.5,128) .. controls (181.13,128) and (163,109.87) .. (163,87.5) -- cycle ;
\draw   (99.25,193.5) .. controls (99.25,171.13) and (117.38,153) .. (139.75,153) .. controls (162.12,153) and (180.25,171.13) .. (180.25,193.5) .. controls (180.25,215.87) and (162.12,234) .. (139.75,234) .. controls (117.38,234) and (99.25,215.87) .. (99.25,193.5) -- cycle ;
\draw   (226.75,193.5) .. controls (226.75,171.13) and (244.88,153) .. (267.25,153) .. controls (289.62,153) and (307.75,171.13) .. (307.75,193.5) .. controls (307.75,215.87) and (289.62,234) .. (267.25,234) .. controls (244.88,234) and (226.75,215.87) .. (226.75,193.5) -- cycle ;
\draw  [fill=uuuuuu, fill opacity=0.3]  (222.5,92) -- (198.11,105) -- (179.5,78) -- (222.5,92) ;
\draw [fill=uuuuuu] (222.5,92) circle (1.2pt);
\draw [fill=uuuuuu]  (198.11,105) circle (1.2pt);
\draw [fill=uuuuuu] (179.5,78) circle (1.2pt);
%
\draw  [fill=uuuuuu, fill opacity=0.3]  (222.5,92) -- (198.11,105) -- (191.5,68) --  (222.5,92) ;
\draw [fill=uuuuuu] (191.5,68) circle (1.2pt);
%
\draw [dashed, fill=uuuuuu, fill opacity=0.3]  (222.5,92) -- (198.11,105) -- (260.5,206) -- (222.5,92);
\draw  [dashed, fill=uuuuuu, fill opacity=0.3]  (222.5,92) -- (198.11,105) -- (279.5,198) -- (222.5,92);
\draw [fill=uuuuuu]  (260.5,206) circle (1.2pt);
\draw [fill=uuuuuu] (279.5,198) circle (1.2pt);
%
\draw  [fill=uuuuuu, fill opacity=0.3] (222.5,92) -- (198.11,105) -- (139.75,193.5) -- (222.5,92) ;
\draw [fill=uuuuuu] (139.75,193.5) circle (1.2pt);
%
\draw   (435,87) .. controls (435,64.63) and (453.13,46.5) .. (475.5,46.5) .. controls (497.87,46.5) and (516,64.63) .. (516,87) .. controls (516,109.37) and (497.87,127.5) .. (475.5,127.5) .. controls (453.13,127.5) and (435,109.37) .. (435,87) -- cycle ;
\draw   (371.25,193) .. controls (371.25,170.63) and (389.38,152.5) .. (411.75,152.5) .. controls (434.12,152.5) and (452.25,170.63) .. (452.25,193) .. controls (452.25,215.37) and (434.12,233.5) .. (411.75,233.5) .. controls (389.38,233.5) and (371.25,215.37) .. (371.25,193) -- cycle ;
\draw   (498.75,193) .. controls (498.75,170.63) and (516.88,152.5) .. (539.25,152.5) .. controls (561.62,152.5) and (579.75,170.63) .. (579.75,193) .. controls (579.75,215.37) and (561.62,233.5) .. (539.25,233.5) .. controls (516.88,233.5) and (498.75,215.37) .. (498.75,193) -- cycle ;
\draw  [fill=uuuuuu, fill opacity=0.3]  (475.5,87) -- (523.05,198.9) -- (556.5,206.54) -- (475.5,87) ;
\draw [fill=uuuuuu] (475.5,87) circle (1.2pt);
\draw [fill=uuuuuu]  (523.05,198.9) circle (1.2pt);
\draw [fill=uuuuuu] (556.5,206.54) circle (1.2pt);
%
\draw  [fill=uuuuuu, fill opacity=0.3]  (475.5,87) -- (523.05,198.9) -- (554.47,183.62) -- (475.5,87);
\draw [fill=uuuuuu] (554.47,183.62) circle (1.2pt);
%
\draw  [dashed, fill=uuuuuu, fill opacity=0.3]  (475.5,87) -- (523.05,198.9) -- (405.5,191) -- (475.5,87) ;
\draw  [dashed, fill=uuuuuu, fill opacity=0.3]  (475.5,87) -- (523.05,198.9) -- (414.5,208) -- (475.5,87);
\draw [fill=uuuuuu] (405.5,191) circle (1.2pt);
\draw [fill=uuuuuu] (414.5,208) circle (1.2pt);
%
\draw (192,24) node [anchor=north west][inner sep=0.75pt]   [align=left] {$V_{1}$};
\draw (262,239) node [anchor=north west][inner sep=0.75pt]   [align=left] {$V_{2}$};
\draw (127,239) node [anchor=north west][inner sep=0.75pt]   [align=left] {$V_{3}$};
\draw (465,24) node [anchor=north west][inner sep=0.75pt]   [align=left] {$V_{1}$};
\draw (533,239) node [anchor=north west][inner sep=0.75pt]   [align=left] {$V_{2}$};
\draw (399,239) node [anchor=north west][inner sep=0.75pt]   [align=left] {$V_{3}$};
\end{tikzpicture}
\caption{Bad and missing edges in Lemmas~\ref{LEMMA:K43-L2-improvement-phase-one} and~\ref{LEMMA:K43-L2-improvement-phase-two}.}
\label{fig:local-adjust}
\end{figure}
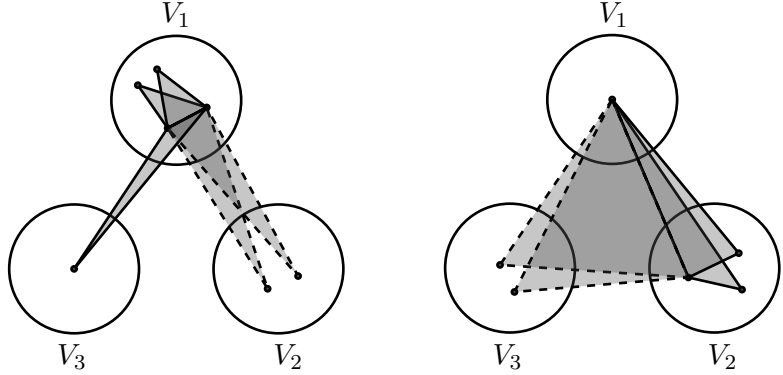

\begin{lemma}\label{LEMMA:K43-L2-improvement-phase-one}
    There exists a constant $\xi_{\ref{LEMMA:K43-L2-improvement-phase-one}}>0$ such that the following holds for every $\xi \in (0, \xi_{\ref{LEMMA:K43-L2-improvement-phase-one}})$ and every $n \ge 1/\xi$. 
    Let $\mathcal{G}$ be an $n$-vertex $3$-graph, and let $V_{1} \cup V_{2} \cup V_{3} = V(\mathcal{G})$ be a partition. 
    Let $e_{\ast} \in \partial\mathcal{G}$ be an edge in the shadow such that $e_{\ast} \in \overline{K}[V_1, V_2, V_3]$. 
    Suppose that 
    \begin{enumerate}[label=(\roman*)]
        \item\label{LEMMA:K43-L2-improvement-phase-one-1} $\max_{i \in [3]}|V_i - n/3| \le \xi n$, 
        \item\label{LEMMA:K43-L2-improvement-phase-one-2} $\max\left\{\Delta(\mathcal{M}),~\Delta(\mathcal{B})\right\} \le \xi n^2$, 
        \item\label{LEMMA:K43-L2-improvement-phase-one-3} $d_{\mathcal{M}}(e_{\ast}) \ge 47\xi^{1/2} n$, 
        \item\label{LEMMA:K43-L2-improvement-phase-one-4} $d_{\mathcal{M}}(e_{\ast}) \ge d_{\mathcal{B}}(e_{\ast}) - \xi n$, 
        \item\label{LEMMA:K43-L2-improvement-phase-one-5} $d_{\mathcal{B}^{\mathrm{bi}}}(e_{\ast}) \le \xi n$. 
    \end{enumerate}
    Then the new $3$-graph $\mathcal{G}^{\ast}\coloneqq \left( \mathcal{G} \setminus \mathcal{B}(e_{\ast}) \right) \cup \mathcal{M}(e_{\ast})$ satisfies 
    \begin{align*}
        \norm{\mathcal{G}^{\ast}}_{2}
        > \norm{\mathcal{G}}_{2}. 
    \end{align*}
\end{lemma}
\begin{proof}[Proof of Lemma~\ref{LEMMA:K43-L2-improvement-phase-one}]
    Let $\mathcal{G}$, $(V_1, V_2, V_3)$, and $e_{\ast}$ be as in the statement of the lemma.  
    By symmetry, we may assume that $e_{\ast} = \{u_1, u_2\} \subseteq V_1$.
    Let 
    \begin{align*}
        \mathcal{S}_{1}
        & \coloneqq \big\{ u_1 w \colon w\in N_{\mathcal{M}}(e_{\ast}) \big\}  \cup \big\{ u_2 w \colon w\in N_{\mathcal{M}}(e_{\ast}) \big\}, \\
        \mathcal{S}_{2}
        & \coloneqq \big\{ u_1 w \colon w\in N_{\mathcal{B}}(e_{\ast}) \cap V_1 \big\}  \cup \big\{ u_2 w \colon w\in N_{\mathcal{B}}(e_{\ast}) \cap V_1 \big\} \\
        \mathcal{S}_{3}
        & \coloneqq \big\{ u_1 w \colon w\in N_{\mathcal{B}}(e_{\ast}) \cap V_3 \big\}  \cup \big\{ u_2 w \colon w\in N_{\mathcal{B}}(e_{\ast}) \cap V_3 \big\}.
    \end{align*}
    Notice that 
    \begin{align*}
        |\mathcal{S}_{1}| = 2d_{\mathcal{M}}(e_{\ast}), \quad
        |\mathcal{S}_{2}| + |\mathcal{S}_{3}| = 2d_{\mathcal{B}}(e_{\ast}), \quad 
        |\mathcal{S}_{3}| = 2d_{\mathcal{B}^{\mathrm{bi}}}(e_{\ast}). 
    \end{align*}

    Let us consider the difference $\norm{\mathcal{G}^{\ast}}_{2} - \norm{\mathcal{G}}_{2}$. 
    Note that for each $e \in \mathcal{S}_{1}$, the codegree of $e$ increased by one, and for each $e \in \mathcal{S}_{2} \cup \mathcal{S}_{3}$, the the codegree of $e$ decreased by one. 
    Therefore, 
    \begin{align}\label{equ:local-adj-one-change}
        \norm{\mathcal{G}^{\ast}}_{2} - \norm{\mathcal{G}}_{2}
        & = \sum_{e \in \binom{V}{2}} \big(d_{\mathcal{G}^{\ast}}^{2}(e) - d_{\mathcal{G}}^{2}(e) \big) \notag \\
        & = \big(d_{\mathcal{G}^{\ast}}^{2}(e_{\ast}) - d_{\mathcal{G}}^{2}(e_{\ast}) \big) + \sum_{e \in \mathcal{S}_{1}} \big(d_{\mathcal{G}^{\ast}}^{2}(e) - d_{\mathcal{G}}^{2}(e) \big) + \sum_{e \in \mathcal{S}_{2} \cup \mathcal{S}_{3}} \big(d_{\mathcal{G}^{\ast}}^{2}(e) - d_{\mathcal{G}}^{2}(e) \big) \notag \\
        & \ge \big(d_{\mathcal{G}^{\ast}}(e_{\ast}) - d_{\mathcal{G}}(e_{\ast}) \big)\big(d_{\mathcal{G}^{\ast}}(e_{\ast}) + d_{\mathcal{G}}(e_{\ast}) \big) + \sum_{e \in \mathcal{S}_{1}} 2 d_{\mathcal{G}}(e) - \sum_{e \in \mathcal{S}_{2} \cup \mathcal{S}_{3}} 2 d_{\mathcal{G}}(e) \notag \\
        & = \big(d_{\mathcal{M}}(e_{\ast}) - d_{\mathcal{B}}(e_{\ast}) \big)\big(d_{\mathcal{G}^{\ast}}(e_{\ast}) + d_{\mathcal{G}}(e_{\ast}) \big) + \sum_{e \in \mathcal{S}_{1}} 2 d_{\mathcal{G}}(e) - \sum_{e \in \mathcal{S}_{2}\cup \mathcal{S}_{3}} 2 d_{\mathcal{G}}(e) \notag \\
        & \ge - 2 \xi n^2 + \sum_{e \in \mathcal{S}_{1}} 2 d_{\mathcal{G}}(e) - \sum_{e \in \mathcal{S}_{2} \cup \mathcal{S}_{3}} 2 d_{\mathcal{G}}(e), 
    \end{align}
    where the last inequality follows from part~\ref{LEMMA:K43-L2-improvement-phase-one-4} and the trivial upper bound $d_{\mathcal{G}^{\ast}}(e_{\ast}) + d_{\mathcal{G}}(e_{\ast}) \le n + n = 2n$. 

    Let $\mathcal{S}_{1}^{h} \coloneqq \left\{e \in \mathcal{S}_{1} \colon d_{\mathcal{M}}(e) \ge \xi^{1/2} n \right\}$. 
    Note that  
    \begin{align*}
        d_{\mathcal{M}}(u_1) + d_{\mathcal{M}}(u_2)
        \ge \frac{1}{2} \sum_{e \in \mathcal{S}_{1}^{h}} d_{\mathcal{M}}(e) 
        \ge \frac{1}{2} |\mathcal{S}_{1}^{h}| \cdot \xi^{1/2} n. 
    \end{align*}
    So it follows from $\Delta(\mathcal{M}) \le \xi n^2$ that 
    \begin{align*}
        |\mathcal{S}_{1}^{h}|
        \le \frac{2 \big( d_{\mathcal{M}}(u_1) + d_{\mathcal{M}}(u_2) \big)}{\xi^{1/2} n}
        \le \frac{4\xi n^2}{\xi^{1/2} n}
        = 4\xi^{1/2} n. 
    \end{align*}
    Combining this with part~\ref{LEMMA:K43-L2-improvement-phase-one-1} and the definition of $\mathcal{S}_{1}^{h}$, we obtain 
    \begin{align}\label{equ:local-adj-one-term-a}
        \sum_{e \in \mathcal{S}_{1}} d_{\mathcal{G}}(e)
        \ge \sum_{e \in \mathcal{S}_{1}\setminus \mathcal{S}_{1}^{h}} d_{\mathcal{G}}(e) 
        & \ge \sum_{e \in \mathcal{S}_{1}\setminus \mathcal{S}_{1}^{h}} \big( d_{\mathbb{C}[V_1, V_2, V_3]}(e) - d_{\mathcal{M}}(e) \big) \notag \\ 
        & \ge |\mathcal{S}_{1}\setminus \mathcal{S}_{1}^{h}| \left(|V_{1}|+|V_{3}|-1 - \xi^{1/2} n \right) \notag \\
        & \ge \left(|\mathcal{S}_{1}| - 4 \xi^{1/2} n\right)\left(\frac{2n}{3}- 4 \xi^{1/2} n \right) \notag \\
        & \ge \frac{2 n |\mathcal{S}_{1}|}{3} - \left(\frac{2 n}{3} +|\mathcal{S}_{1}|\right) \cdot 4\xi^{1/2} n 
        > \frac{2n |\mathcal{S}_{1}|}{3}  - 12 \xi^{1/2} n^2. 
    \end{align}
    
    Similarly, let $\mathcal{S}_{2}^{h} \coloneqq \left\{e \in \mathcal{S}_{2} \colon d_{\mathcal{B}}(e) \ge \xi^{1/2} n \right\}$. By the same argument as above, the bound $\Delta(\mathcal{B}) \le \xi n^2$ implies that $|\mathcal{S}_{2}^{h}| \le 4 \xi^{1/2} n$. 
    It follows from part~\ref{LEMMA:K43-L2-improvement-phase-one-5} that $|\mathcal{S}_{3}| \le 2 \xi n$. 
    Therefore, 
    \begin{align}\label{equ:local-adj-one-term-b}
        \sum_{e \in \mathcal{S}_{2}\cup \mathcal{S}_{3}} d_{\mathcal{G}}(e)
        & = \sum_{e \in \mathcal{S}_{2}\setminus \mathcal{S}_{2}^{h}} d_{\mathcal{G}}(e) + \sum_{e \in \mathcal{S}_{2}^{h} \cup \mathcal{S}_{3}} d_{\mathcal{G}}(e) \notag \\
        & \le |\mathcal{S}_{2}| \left(d_{\mathbb{C}[V_1, V_2, V_3]} + \xi^{1/2} n \right) + \big( |\mathcal{S}_{2}^{h}| + |\mathcal{S}_{3}| \big) n \notag \\
        & \le |\mathcal{S}_{2}| \left(\frac{n}{3} + \xi n + \xi^{1/2} n \right) + \big( 4 \xi^{1/2} n + 2 \xi n \big) n
        < \frac{n |\mathcal{S}_{2}|}{3}  + 10\xi^{1/2} n^2. 
    \end{align}
    Combining~\eqref{equ:local-adj-one-change},~\eqref{equ:local-adj-one-term-a}, and~\eqref{equ:local-adj-one-term-b}, and applying parts~\ref{LEMMA:K43-L2-improvement-phase-one-3} and~\ref{LEMMA:K43-L2-improvement-phase-one-4}, we obtain 
    \begin{align*}
        \norm{\mathcal{G}^{\ast}}_{2} - \norm{\mathcal{G}}_{2}
        & \ge - 2 \xi n^2 + 2\left(\frac{2n |\mathcal{S}_{1}|}{3}  - 12 \xi^{1/2} n^2\right) - 2\left(\frac{n |\mathcal{S}_{2}|}{3}  + 10\xi^{1/2} n^2 \right) \\
        & \ge \frac{4n}{3} \left(\frac{|\mathcal{S}_{1}|}{2} + \frac{|\mathcal{S}_{1}|}{2} - \frac{|\mathcal{S}_{2}|}{2}\right) - 46 \xi^{1/2} n^2 \\
        & = \frac{4n}{3} \left(d_{\mathcal{M}}(e^{\ast}) + d_{\mathcal{M}}(e^{\ast}) - d_{\mathcal{B}}(e^{\ast}) \right) - 46 \xi^{1/2} n^2 \\
        & \ge \frac{4n}{3} \left(47\xi^{1/2} n - \xi n \right) - 46 \xi^{1/2} n^2 
        > 0, 
    \end{align*}
    completing the proof of Lemma~\ref{LEMMA:K43-L2-improvement-phase-one}. 
\end{proof}

\begin{lemma}\label{LEMMA:K43-L2-improvement-phase-two}
    There exists $\xi_{\ref{LEMMA:K43-L2-improvement-phase-two}}>0$ such that the following holds for for every $\xi \in (0, \xi_{\ref{LEMMA:K43-L2-improvement-phase-two}})$ and every $n \ge 1/\xi$. 
    Let $\mathcal{G}$ be an $n$-vertex $3$-graph and let $V_{1} \cup V_{2} \cup V_{3} = V(\mathcal{G})$ be a partition. Let $e_{\ast} \in \partial\mathcal{H}$ be an edge in the shadow such that $e_{\ast} \in K[V_1, V_2, V_3]$. . 
    Suppose that  
    \begin{enumerate}[label=(\roman*)]
        \item\label{LEMMA:K43-L2-improvement-phase-two-1} $\max_{i \in [3]}|V_i - n/3| \le \xi n$, 
        \item\label{LEMMA:K43-L2-improvement-phase-two-2} $\max\left\{\Delta(\mathcal{M}),~\Delta(\mathcal{B}) \right\}\le \xi n^2$, 
        \item\label{LEMMA:K43-L2-improvement-phase-two-3} $d_{\mathcal{M}^{\mathrm{tri}}}(e_{\ast}) \ge 90\xi^{1/2} n$, 
        \item\label{LEMMA:K43-L2-improvement-phase-two-4} $d_{\mathcal{M}^{\mathrm{tri}}}(e_{\ast}) \ge d_{\mathcal{B}}(e_{\ast}) - \xi n$. 
    \end{enumerate}
    Then the new $3$-graph $\mathcal{G}^{\ast}\coloneqq \left( \mathcal{G} \setminus \mathcal{B}(e_{\ast}) \right) \cup \mathcal{M}(e_{\ast})$ satisfies 
    \begin{align*}
        \norm{\mathcal{G}^{\ast}}_{2}
        > \norm{\mathcal{G}}_{2}. 
    \end{align*}
\end{lemma}
\begin{proof}[Proof of Lemma~\ref{LEMMA:K43-L2-improvement-phase-two}]
    Let $\mathcal{G}$, $(V_1, V_2, V_3)$, and $e_{\ast} = \{u_1, u_2\}$ be as in the statement of the lemma.  
    By symmetry, we may assume that $(u_1, u_2) \in V_1 \times V_2$.
    Let 
    \begin{align*}
        \mathcal{S}_{1}
        & \coloneqq \big\{ u_1 w \colon w\in N_{\mathcal{M}^{\mathrm{tri}}}(e_{\ast}) \big\}  \cup \big\{ u_2 w \colon w\in N_{\mathcal{M}^{\mathrm{tri}}}(e_{\ast}) \big\}, \\
        \mathcal{S}_{2a}
        & \coloneqq \big\{ u_1 w \colon w\in N_{\mathcal{B}}(e_{\ast}) \big\}  
        \quad\text{and}\quad 
        \mathcal{S}_{2b} 
        \coloneqq \big\{ u_2 w \colon w\in N_{\mathcal{B}}(e_{\ast}) \big\}.
    \end{align*}
    Let $\mathcal{S}_{2} \coloneqq \mathcal{S}_{2a} \cup \mathcal{S}_{2b}$. 
    Notice that 
    \begin{align*}
        |\mathcal{S}_{1}| = 2 d_{\mathcal{M}^{\mathrm{tri}}}(e_{\ast}), \quad 
        |\mathcal{S}_{2a}| = |\mathcal{S}_{2b}| = d_{\mathcal{B}}(e_{\ast}). 
    \end{align*}
    %
    By the same argument as in Lemma~\ref{LEMMA:K43-L2-improvement-phase-one}, we have 
    \begin{align}\label{equ:local-adj-two-change}
        \norm{\mathcal{G}^{\ast}}_{2} - \norm{\mathcal{G}}_{2}
        & \ge \big(d_{\mathcal{M}^{\mathrm{tri}}}(e_{\ast}) - d_{\mathcal{B}}(e_{\ast}) \big)\big(d_{\mathcal{G}^{\ast}}(e_{\ast}) + d_{\mathcal{G}}(e_{\ast}) \big) + \sum_{e \in \mathcal{S}_{1}} 2 d_{\mathcal{G}}(e) - \sum_{e \in \mathcal{S}_{2}} 2 d_{\mathcal{G}}(e) \notag \\
        & \ge - 2 \xi n^2 + \sum_{e \in \mathcal{S}_{1}} 2 d_{\mathcal{G}}(e) - \sum_{e \in \mathcal{S}_{2}} 2 d_{\mathcal{G}}(e). 
    \end{align}
    Let $\mathcal{S}_{1}^{h} \coloneqq \left\{e \in \mathcal{S}_{1} \colon d_{\mathcal{M}^{\mathrm{tri}}}(e) \ge \xi^{1/2} n \right\}$ and $\mathcal{S}_{2}^{h} \coloneqq \left\{e \in \mathcal{S}_{2} \colon d_{\mathcal{B}}(e) \ge \xi^{1/2} n \right\}$. 
    Similar to the proof of Lemma~\ref{LEMMA:K43-L2-improvement-phase-one}, it follows from part~\ref{LEMMA:K43-L2-improvement-phase-two-2} that $\max\{|\mathcal{S}_{1}^{h}|,~|\mathcal{S}_{2}^{h}|\} \le 4\xi^{1/2} n$. 
    Similarly, we have 
    \begin{align*}
        \sum_{e \in \mathcal{S}_{1}} d_{\mathcal{G}}(e)
        > \frac{2n |\mathcal{S}_{1}|}{3}  - 12 \xi^{1/2} n^2, 
    \end{align*}
    and 
    \begin{align*}
        \sum_{e \in \mathcal{S}_{2}} d_{\mathcal{G}}(e)
        & = \sum_{e \in \mathcal{S}_{2} \setminus \mathcal{S}_{2}^{h}} d_{\mathcal{G}}(e) + \sum_{e\in \mathcal{S}_{2}^{h}} d_{\mathcal{G}}(e) \\
        & \le \sum_{e \in \mathcal{S}_{2} \setminus \mathcal{S}_{2}^{h}} \left(d_{\mathbb{C}[V_1, V_2, V_3]}(e) + \xi^{1/2} n \right)  + |\mathcal{S}_{2}^{h}| \cdot n \\
        & \le \sum_{e \in \mathcal{S}_{2a} \setminus \mathcal{S}_{2}^{h}}\left(\frac{2n}{3} + 2\xi + \xi^{1/2} n \right) + \sum_{e \in \mathcal{S}_{2b} \setminus \mathcal{S}_{2}^{h}}\left(\frac{n}{3} + \xi + \xi^{1/2} n \right) + 4\xi^{1/2} n^2 \\
        & \le \frac{2n|\mathcal{S}_{2a}|}{3} + \frac{n|\mathcal{S}_{2b}|}{3} + 9\xi^{1/2} n^2.  
    \end{align*}
    Combining these two inequalities with~\eqref{equ:local-adj-two-change}, and applying parts~\ref{LEMMA:K43-L2-improvement-phase-two-3} and~\ref{LEMMA:K43-L2-improvement-phase-two-4}, we obtain 
    \begin{align*}
        \norm{\mathcal{G}^{\ast}}_{2} - \norm{\mathcal{G}}_{2}
        & \ge - 2 \xi n^2 + 2\left(\frac{2n |\mathcal{S}_{1}|}{3}  - 12 \xi^{1/2} n^2\right) - 2\left(\frac{2n|\mathcal{S}_{2a}|}{3} + \frac{n|\mathcal{S}_{2b}|}{3} + 9\xi^{1/2} n^2\right) \\
        & = \frac{4n}{3}\left( \frac{|\mathcal{S}_{1}|}{4} + \frac{|\mathcal{S}_1|}{2} - |\mathcal{S}_{2a}| + \frac{|\mathcal{S}_{1}|}{4} - \frac{|\mathcal{S}_{2b}|}{2} \right) - 44 \xi^{1/2} n^2  \\
         & = \frac{4n}{3}\left( \frac{d_{\mathcal{M}^{\mathrm{tri}}}(e_{\ast})}{2}  + \frac{3}{4}\left(d_{\mathcal{M}^{\mathrm{tri}}}(e_{\ast}) - d_{\mathcal{B}}(e_{\ast})\right) \right) - 44 \xi^{1/2} n^2  \\
        & \ge \frac{4n}{3}\left( \frac{90\xi^{1/2} n^2}{2}  - \frac{3\xi n^2}{4} \right) - 44 \xi^{1/2} n^2
        > 0.
    \end{align*}
    This completes the proof of Lemma~\ref{LEMMA:K43-L2-improvement-phase-two}. 
\end{proof}

\subsection{Proof of Theorem~\ref{THM:L2-exact-K43}}\label{SUBSEC:proof-L2-exact-K43}
In this subsection, we present the proof of Theorem~\ref{THM:L2-exact-K43}. 
\begin{proof}[Proof of Theorem~\ref{THM:L2-exact-K43}]
    Let $\delta, \delta_{1}, \delta_{2}, \delta_3, \delta_4 > 0$ be sufficiently small constants such that $\delta \ll \delta_{1} \ll \delta_{2} \ll \delta_3 \ll \delta_4 \ll 1$. 
    Let $n$ be a sufficiently large integer. 
    Let $\mathcal{H}$ be an $n$-vertex $K_{4}^{3}$-free $3$-graph on vertex set $V$ with $\norm{\mathcal{H}}_{2} = \mathrm{ex}_{\ell_2}(n,K_{4}^{3})$, that is, with the maximum possible $\ell_{2}$-norm. 
    Let $V_1 \cup V_2 \cup V_3 = V$ be a partition that maximizes the size of the intersection 
    \begin{align*}
        \mathcal{H}[V_1, V_2, V_3] \coloneqq \mathcal{H} \cap \mathbb{C}[V_1, V_2, V_3]. 
    \end{align*}
    Recall that $K[V_1, V_2, V_3]$ is the complete $3$-partite graph with parts $V_1$, $V_2$, and $V_3$. 
    Also, recall the definitions of $\mathcal{B}$, $\mathcal{M}$, $\mathcal{B}^{\mathrm{int}}$, $\mathcal{B}^{\mathrm{bi}}$, $\mathcal{M}^{\mathrm{tri}}$, and $\mathcal{M}^{\mathrm{bi}}$ from~\eqref{equ:def-bad-missing-triples},~\eqref{equ:def-B-int-bi}, and~\eqref{equ:def-M-tri-bi}, respectively. 
    We may assume that $\mathcal{B} \neq \emptyset$, since otherwise we are done. 

    \begin{claim}\label{CLAIM:K43-B-M-upper-bound}
        We have $\max\left\{ |\mathcal{M}|,~|\mathcal{B}| \right\} \le 2\delta_{1} n^3$. 
    \end{claim}
    \begin{proof}[Proof of Claim~\ref{CLAIM:K43-B-M-upper-bound}]
        Since $\mathbb{C}_{n}$ is $K_{4}^{3}$-free, we have 
        \begin{align*}
            \norm{\mathcal{H}}_{2}
            = \mathrm{ex}_{\ell_2}(n,K_{4}^{3})
            \ge \norm{\mathbb{C}_{n}}_{2}
            \ge \frac{n^4}{6} - \delta n^{4}. 
        \end{align*}
        It follows from Theorem~\ref{THM:K43-L2-stability} that $|\mathcal{H} \triangle \mathbb{C}_{n}| \le \delta_1 n^3$. 
        Then the choice of the partition $V_1 \cup V_2 \cup V_3 = V$ ensures that $|\mathcal{H}\cap\mathbb{C}[V_1, V_2, V_3]| \ge |\mathcal{H}\cap\mathbb{C}_n|$, which implies  
        \begin{align}\label{equ:H-cap-C-V1V2V3}
            |\mathcal{H}[V_1, V_2, V_3]| 
            \ge |\mathbb{C}_{n}| - \delta_1 n^3
            \ge \frac{5 n^3}{54}  - 2\delta_1 n^3, 
        \end{align}
        It follows that 
        \begin{align*}
            |\mathcal{M}| 
            = |\mathbb{C}[V_1, V_2, V_3]| - |\mathcal{H}[V_1, V_2, V_3]| 
            \le |\mathbb{C}_{n}| - |\mathcal{H}[V_1, V_2, V_3]| 
            \le 2\delta_{1} n^3. 
        \end{align*}
        Note that $|\mathcal{H} \triangle \mathbb{C}_{n}| \le \delta_1 n^3$ implies $|\mathcal{H}| \le |\mathbb{C}_{n}| + \delta_1 n^3$. 
        Combining it with~\eqref{equ:H-cap-C-V1V2V3}, we obtain 
        \begin{align*}
            |\mathcal{B}|
            = |\mathcal{H}| - |\mathcal{H}[V_1, V_2, V_3]|
            \le |\mathbb{C}_{n}| + \delta_1 n^3 - \left(|\mathbb{C}_{n}| - \delta_1 n^3\right)
            = 2\delta_1 n^3. 
        \end{align*}
        This completes the proof of Claim~\ref{CLAIM:K43-B-M-upper-bound}. 
    \end{proof}

    Let $x_i \coloneqq |V_i|/n$ for $i \in [3]$. 
    \begin{claim}\label{CLAIM:K43-L2-xi-size}
        We have $\max_{i \in [3]}|x_i - 1/3| \le 10\delta_{1}^{1/2}$. 
    \end{claim}
    \begin{proof}[Proof of Claim~\ref{CLAIM:K43-L2-xi-size}]
        It follows from~\eqref{equ:H-cap-C-V1V2V3} that 
        \begin{align*}
            \frac{5}{54} - 2\delta_1 
            \le \frac{|\mathcal{H}[V_1, V_2, V_3]|}{n^3}
            \le \frac{|\mathbb{C}[V_1, V_2, V_3]|}{n^3}
            \le  x_1 x_2 x_3 + \frac{x_1^2 x_2}{2} + \frac{x_2^2 x_3}{2} + \frac{x_3^2 x_1}{2}. 
        \end{align*}
        Combining this with and Lemma~\ref{LEMMA:3-part-inequality}, we obtain 
        \begin{align*}
            \max_{i\in [3]} \frac{1}{50}\left(x_i - \frac{1}{3}\right)^2
            \le 2\delta_{1}, 
        \end{align*}
        which implies that $\max_{i \in [3]}|x_i - 1/3| \le 10\delta_{1}^{1/2}$. 
    \end{proof}

    \begin{claim}\label{CLAIM:max-deg-B}
        We have $\max\left\{ \Delta(\mathcal{B}),~\Delta(\mathcal{M}) \right\} \le 2 \delta_{3} n^2$. 
    \end{claim}
    \begin{proof}[Proof of Claim~\ref{CLAIM:max-deg-B}]
        Fix an arbitrary vertex $v \in V$. By symmetery, we may assume that $v\in V_1$. 
        For convenience, let $L \coloneqq L_{\mathcal{H}}(v)$. 
        Recall from~\eqref{equ:def-2norm-degree-b} that 
        \begin{align}\label{equ:K43-v-2norm-degree}
            s_{\mathcal{H}}(v)
            = \norm{L}_{2} + 2\sum_{e\in L} d_{\mathcal{H}}(e) - |L|. 
        \end{align}
        It follows from Lemma~\ref{LEMMA:2norm-degree-regular} that 
        \begin{align}\label{equ:min-L2-deg}
            s_{\mathcal{H}}(v)
            \ge s(\mathcal{H}) - 60 n^2 
            \ge s(\mathbb{C}_{n}) - 60 n^2 
            \ge \frac{2n^3}{3} - \delta n^3. 
        \end{align}
        It follows from the maximality of $\mathcal{H} \cap \mathbb{C}[V_1, V_2, V_3]$ that moving $v$ from $V_1$ to $V_2$ and $V_3$ will not increase the number of edges in the intersection. 
        Therefore, 
        \begin{align}
            |L[V_1, V_2]| + |L[V_3]| & \ge |L[V_1, V_3]| + |L[V_1]|, \label{equ:K43-proof-local-max-a} \\ 
            |L[V_2, V_3]| + |L[V_3]| & \ge |L[V_1, V_3]| + |L[V_2]|. \label{equ:K43-proof-local-max-b}
        \end{align}
        Recall that $L[V_1, V_2, V_3] \coloneqq K[V_1, V_2, V_2] \cap L$ is the induced $3$-partite subgraph of $L$ with parts $V_1, V_2, V_3$. 
        For convenience, let $\overline{L}[V_1, V_2, V_3] \coloneqq L \setminus L[V_1, V_2, V_3]$. 
        It follows from Claims~\ref{CLAIM:K43-B-M-upper-bound} and~\ref{CLAIM:K43-L2-xi-size} that 
        \begin{align*}
            \sum_{e\in L} d_{\mathcal{H}}(e)
            & \le \sum_{e\in L} \left( d_{\mathbb{C}[V_1, V_2, V_3]}(e) + d_{\mathcal{B}}(e) \right) \\
            & = \sum_{e\in \overline{L}[V_1, V_2, V_3]} d_{\mathbb{C}[V_1, V_2, V_3]}(e) + \sum_{e\in L[V_1, V_2, V_3]} d_{\mathbb{C}[V_1, V_2, V_3]}(e) + \sum_{e\in L} d_{\mathcal{B}}(e) \\
            & \le \left(\frac{n}{3} + 10 \delta_{1}^{1/2} n\right)|L \setminus L[V_1, V_2, V_3]| + 2\left(\frac{n}{3} + 10 \delta_{1}^{1/2} n\right) |L[V_1, V_2, V_3]| + 3|\mathcal{B}| \\
            & = \frac{n}{3}\left(|L| + |L[V_1, V_2, V_3]|\right) + 10 \delta_{1}^{1/2} n \left(|L| + |L[V_1, V_2, V_3]|\right) + 6\delta_1 n^3 \\
            & \le \frac{n}{3}\left(|L| + |L[V_1, V_2, V_3]|\right) + 16 \delta_{1}^{1/2} n^3. 
        \end{align*}
        Combining this with~\eqref{equ:K43-v-2norm-degree} and~\eqref{equ:min-L2-deg}, we obtain 
        \begin{align}\label{equ:K43-link-2norm-1norm-lower-bound}
            \norm{L}_{2} + \frac{2n}{3} \left(|L| + |L[V_1, V_2, V_3]|\right) 
            \ge s_{\mathcal{H}}(v) - 16 \delta_{1}^{1/2} n^3
            \ge \frac{2}{3} n^3 - 17 \delta_{1}^{1/2} n^3. 
        \end{align}
        Let $\mathcal{T}$ be the set of triangles in $L$ of types $V_1V_2V_3$, $V_1V_1V_2$, $V_2V_2V_3$, and $V_3V_3V_1$, that is, 
        \begin{align*}
            \mathcal{T}
            \coloneqq \big\{ S \in \mathbb{C}[V_1, V_2, V_3] \colon L[S] \cong K_{3} \big\}. 
        \end{align*}
        Suppose, for contradiction, that there exists a triple $S \in \mathcal{T} \setminus \mathcal{M}$. 
        Then the set $S \cup \{v\}$ would induce a copy of $K_{4}^{3}$ in $\mathcal{H}$, contradicting the $K_{4}^{3}$-freeness of $\mathcal{H}$. 
        Thus, $\mathcal{T} \subseteq \mathcal{M}$, and consequently, by Claim~\ref{CLAIM:K43-B-M-upper-bound},
        \begin{align}\label{equ:K43-L2-triangles-upper-bound}
            |\mathcal{T}|
            \le |\mathcal{M}|
            \le 2\delta_1 n^3. 
        \end{align}
        Combining this with Claim~\ref{CLAIM:K43-L2-xi-size},~\eqref{equ:K43-proof-local-max-a}, and~\eqref{equ:K43-proof-local-max-b}, it follows from Theorem~\ref{THM:3-colored-Mantel-L2-norm} that 
        \begin{align*}
            \norm{L}_{2}
            \le 9\left(\frac{n}{3}\right)^{3} + \delta_{2} n^3 
            = \frac{n^3}{3} + \delta_2 n^3. 
        \end{align*}
        It then follows from~\eqref{equ:K43-link-2norm-1norm-lower-bound} that 
        \begin{align}\label{equ:K43-link-1norm-lower-bound}
            |L[V_1, V_2, V_3]| + |L|
            & \ge \frac{3}{2n} \left(\frac{2}{3} n^3 - 17 \delta_{1}^{1/2} n^3 - \frac{n^3}{3} - \delta_2 n^3 \right)
            \ge \frac{9}{2}\left(\frac{n}{3}\right)^2 - 3\delta_2 n^2. 
        \end{align}
        Combining this with~\eqref{equ:K43-L2-triangles-upper-bound} and Theorem~\ref{THM:3-colored-Mantel-stability}, we conclude that for some $i \in [3]$,
        \begin{align}\label{equ:K43-bad-links}
            |L[V_{i}]| + |L[V_{i+1}]| + |L[V_{i}, V_{i+2}]|
            \le \delta_{3} n^2. 
        \end{align}
        We claim that $i = 1$. 
        Suppose to the contrary that $i = 2$. 
        Then it follows from~\eqref{equ:K43-proof-local-max-a} that 
        \begin{align*}
            |L[V_1, V_3]| + |L[V_1]|
            \le |L[V_1, V_2]| + |L[V_3]|
            \le \delta_{3} n^2. 
        \end{align*}
        Consequently, 
        \begin{align*}
            |L|
            & = \left(|L[V_{2}]| + |L[V_{3}]| + |L[V_{2}, V_{1}]|\right) + \left(|L[V_1, V_3]| + |L[V_1]|\right) + |L[V_2, V_3]| \\
            & \le \delta_{3} n^2 + \delta_{3} n^2 + \left(\frac{n}{3}+10\delta_{1}^{1/2} n\right)^2 
            < \frac{9}{2}\left(\frac{n}{3}\right)^2 - 3\delta_{2} n^2, 
        \end{align*}
        a contradiction to~\eqref{equ:K43-link-1norm-lower-bound}. 
        Thus, $i \neq 2$. 
        Similarly, using~\eqref{equ:K43-proof-local-max-b}, one can show that $i \neq 3$. 
        Therefore, $i = 1$, and hence, 
        \begin{align*}
            d_{\mathcal{B}}(v)
            = |L[V_{1}]| + |L[V_{2}]| + |L[V_{1}, V_{3}]|
            \le \delta_{3} n^2. 
        \end{align*}
        This proves that $\Delta(\mathcal{B}) \le \delta_{3} n^2$. 

        It remains to show that $\Delta(\mathcal{M}) \le 2 \delta_{3} n^2$.
        It follows from~\eqref{equ:K43-L2-triangles-upper-bound} and Theorem~\ref{THM:3-partite-Mantel} that 
        \begin{align*}
            |L[V_1, V_2, V_3]|
            \le 2 \left(\frac{n}{3}\right)^2 + \delta_2 n^2. 
        \end{align*}
        Combining this with~\eqref{equ:K43-link-1norm-lower-bound}, we obtain 
        \begin{align*}
            |L|
            & \ge \frac{9}{2}\left(\frac{n}{3}\right)^2 - 3\delta_2 n^2 - \left(2 \left(\frac{n}{3}\right)^2 + \delta_2 n^2 \right)
            = \frac{5}{2} \left(\frac{n}{3}\right)^2 - 5\delta_2 n^2. 
        \end{align*}
        It then follows from Claim~\ref{CLAIM:K43-L2-xi-size} and~\eqref{equ:K43-bad-links} that 
        \begin{align*}
            d_{\mathcal{M}}(v)
            & = d_{\mathbb{C}[V_1, V_2, V_3]}(v) - \left(|L[V_1,V_2]|+|L[V_2, V_3]| + |L[V_3]|\right) \\
            & \le d_{\mathbb{C}[V_1, V_2, V_3]}(v) - \big( |L| - \left( |L[V_1]|+|L[V_2]| + |L[V_1,V_3]| \right) \big) \\
            & \le \frac{5}{2}\left(\frac{n}{3} + 10\delta_{1}^{1/2} n\right)^2 - \left( \frac{5}{2} \left(\frac{n}{3}\right)^2 - 5\delta_2 n^2 - \delta_{3} n^2 \right) 
            \le 2 \delta_{3} n^2. 
        \end{align*}
        This proves that $\Delta(\mathcal{M}) \le 2 \delta_{3} n^2$, and completes the proof of Claim~\ref{CLAIM:max-deg-B}. 
    \end{proof}

    \begin{claim}\label{CLAIM:K43-bad-contains-heavy}
        Every bad edge $E \in \mathcal{B}$ contains a pair $e\subseteq E$ such that $d_{\mathcal{M}}(e) \ge n/10$. 
        Moreover, if $E = \{x,y,z\} \in \mathcal{B}^{\mathrm{bi}}$ and $(x,y,z)\in V_{i-1} \times V_i \times V_i$ for some $i \in [3]$, then 
        \begin{align*}
            d_{\mathcal{M}}(yz) \ge \frac{n}{10}
            \quad\text{or}\quad 
            \max\big\{ d_{\mathcal{M}^{\mathrm{tri}}}(xy),~d_{\mathcal{M}^{\mathrm{tri}}}(xz) \big\} \ge \frac{n}{10}. 
        \end{align*}
    \end{claim}
    \begin{proof}[Proof of Claim~\ref{CLAIM:K43-bad-contains-heavy}]
        Fix a bad edge $E = \{x,y,z\} \in \mathcal{B}$. 
        First, we consider the case that $E \in \mathcal{B}^{\mathrm{int}}$. 
        By symmetry, we may assume that $E \subseteq V_1$.
        Note that for every $w \in V_2$, it follows from the $K_{4}^{3}$-freeness of $\mathcal{H}$ that $\{wxy, wyz, wxz\} \cap \mathcal{M} \neq \emptyset$. 
        Therefore, 
        \begin{align*}
            d_{\mathcal{M}}(xy) + d_{\mathcal{M}}(yz) + d_{\mathcal{M}}(xz) 
            \ge |V_2|. 
        \end{align*}
        Combining it with Claim~\ref{CLAIM:max-deg-B}, we obtain 
        \begin{align*}
            \max\{d_{\mathcal{M}}(xy),~d_{\mathcal{M}}(yz),~d_{\mathcal{M}}(xz)\}
            \ge \frac{|V_2|}{3}
            \ge \frac{1}{3}\left(\frac{n}{3} - 10 \delta_{1}^{1/2} n\right)
            > \frac{n}{10}. 
        \end{align*}

        Next, we consider the case that $E \in \mathcal{B}^{\mathrm{bi}}$. 
        By symmetry, we may assume that $(x,y,z) \in V_{1} \times V_2 \times V_2$. 
        Note that for every $w \in V_3$, it follows from the $K_{4}^{3}$-freeness of $\mathcal{H}$ that $\{wxy, wyz, wxz\} \cap \mathcal{M} \neq \emptyset$. 
        Therefore, 
        \begin{align*}
            d_{\mathcal{M}^{\mathrm{tri}}}(xy)  + d_{\mathcal{M}^{\mathrm{tri}}}(xz) + d_{\mathcal{M}}(yz)
            \ge |V_3|. 
        \end{align*}
        Combining it with Claim~\ref{CLAIM:max-deg-B}, we obtain 
        \begin{align*}
            \max\{d_{\mathcal{M}^{\mathrm{tri}}}(xy),~d_{\mathcal{M}^{\mathrm{tri}}}(xz),~d_{\mathcal{M}}(yz)\}
            \ge \frac{|V_3|}{3}
            \ge \frac{1}{3}\left(\frac{n}{3} - 10 \delta_{1}^{1/2} n\right)
            > \frac{n}{10}. 
        \end{align*}
        This completes the proof of Claim~\ref{CLAIM:K43-bad-contains-heavy}. 
    \end{proof}

    Recall that $\overline{K}[V_1, V_2, V_3] = \binom{V_1}{2} \cup \binom{V_{2}}{2} \cup \binom{V_{3}}{2}$. 
    \begin{claim}\label{CLAIM:K43-B2-codegree-small}
        For every $e \in \overline{K}[V_1, V_2, V_3]$, we have $d_{\mathcal{B}^{\mathrm{bi}}}(e) \le 3 \delta_{3}^{1/2} n$ and, in particular, 
        \begin{align*}
            d_{\mathcal{B}}(e) 
            \le \max_{i\in [3]}\{|V_i|\} + 3 \delta_{3}^{1/2} n
            \le \frac{n}{3} + 4 \delta_{3}^{1/2} n. 
        \end{align*}
    \end{claim}
    \begin{proof}[Proof of Claim~\ref{CLAIM:K43-B2-codegree-small}]
        Suppose to the contrary that there exists $e = \{u_1, u_2\} \in \overline{K}[V_1, V_2, V_3]$ such that $d_{\mathcal{B}^{\mathrm{bi}}}(e) > 3 \delta_{3}^{1/2} n$. 
        By symmetry, we may assume that $e \subseteq V_1$. 
        Let $N \coloneqq N_{\mathcal{H}}(e) \cap V_3$. 
        It follows from the assumption that $|N| = d_{\mathcal{B}^{\mathrm{bi}}}(e) > 3 \delta_{3}^{1/2} n$. 
        For every pair $\{y,z\} \subseteq N$, it follows from the $K_{4}^{3}$-freeness of $\mathcal{H}$ that $\{u_1yz, u_2yz\} \cap \mathcal{M} \neq \emptyset$.
        Therefore,  
        \begin{align*}
            \max\{d_{\mathcal{M}}(u_1),~d_{\mathcal{M}}(u_2)\}
            \ge \frac{1}{2} \left(d_{\mathcal{M}}(u_1) + d_{\mathcal{M}}(u_2) \right)
            \ge \frac{1}{2}\binom{|N|}{2}
            > 2\delta_{3} n^2, 
        \end{align*}
        a contradiction to Claim~\ref{CLAIM:max-deg-B}. 
    \end{proof}

    \begin{claim}\label{CLAIM:K43-Mcodeg-Bcode-Kbar}
        For every $e \in \overline{K}[V_1, V_2, V_3]$, we have $d_{\mathcal{M}}(e) \ge d_{\mathcal{B}}(e) - 7 \delta_{3}^{1/2} n$ . 
    \end{claim}
    \begin{proof}[Proof of Claim~\ref{CLAIM:K43-Mcodeg-Bcode-Kbar}]
        Suppose to the contrary that there exists $e = \{u_1, u_2\} \in \overline{K}[V_1, V_2, V_3]$ such that this claim fails. 
        By symmetry, we may assume that $e \subseteq V_1$. 
        
        Let $N_{1} \coloneqq N_{\mathcal{H}}(e) \cap V_1$ and $N_{2} \coloneqq N_{\mathcal{H}}(e) \cap V_2$. 
        We may assume that $\min\{|N_1|,~|N_2|\} > 2 \delta_{3}^{1/2} n$. 
        Indeed, if $|N_1| \le 2 \delta_{3}^{1/2} n$, then it follows from Claim~\ref{CLAIM:K43-B2-codegree-small} that 
        \begin{align*}
            d_{\mathcal{B}}(e) 
            = d_{\mathcal{B}^{\mathrm{int}}}(e) + d_{\mathcal{B}^{\mathrm{bi}}}(e)
            = |N_1| + d_{\mathcal{B}^{\mathrm{bi}}}(e)
            \le 5\delta_{3}^{1/2} n, 
        \end{align*}
        which is clearly at most $d_{\mathcal{M}}(e) + 7 \delta_{3}^{1/2} n$. 
        
        Similarly, if $|N_2| \le 2 \delta_{3}^{1/2} n$, then it follows from Claim~\ref{CLAIM:max-deg-B} and Claim~\ref{CLAIM:K43-B2-codegree-small} that 
        \begin{align*}
            d_{\mathcal{M}}(e) 
            = |V_2| - |N_2|
            \ge \frac{n}{3} - 10 \delta_{1}^{1/2} n - 2 \delta_{3}^{1/2} n
            \ge \frac{n}{3} - 3 \delta_{3}^{1/2} n
            \ge d_{\mathcal{B}}(e) - 7 \delta_{3}^{1/2} n. 
        \end{align*}
        Therefore, we may assume that $\min\{|N_1|,~|N_2|\} > 2 \delta_{3}^{1/2} n$.

        Fix an arbitrary pair $(y,z) \in N_1 \times N_2$. It follows from the $K_{4}^{3}$-freeness of $\mathcal{H}$ that $\{u_1yz, u_2yz\} \cap \mathcal{M} \neq \emptyset$. 
        Therefore, 
        \begin{align*}
            \max\{d_{\mathcal{M}}(u_1),~d_{\mathcal{M}}(u_2)\}
            & \ge \frac{1}{2} \left(d_{\mathcal{M}}(u_1) + d_{\mathcal{M}}(u_2) \right) \\
            & \ge \frac{1}{2}|N_1||N_2|
            > \frac{1}{2} \cdot 2 \delta_{3}^{1/2} n \cdot 2 \delta_{3}^{1/2} n
            = 2\delta_{3} n^2, 
        \end{align*}
        a contradiction to Claim~\ref{CLAIM:max-deg-B}. 
    \end{proof}

    Define a subfamily of $\mathcal{B}^{\mathrm{bi}}$ as follows: 
    \begin{align*}
        \tilde{\mathcal{B}}
        \coloneqq \big\{\{u, v, w\}\in \mathcal{B}^{\mathrm{bi}} \colon \text{$(u,v,w) \in V_{i}\times V_{i} \times V_{i-1}$ for some $i \in [3]$ and $d_{\mathcal{M}}(uv) \le \delta_{4} n$} \big\}. 
    \end{align*}

    \begin{claim}\label{CLAIM:K43-Mcodeg-Bcode-K}
        For every $e \in K[V_1, V_2, V_3]$, we have $d_{\mathcal{M}^{\mathrm{tri}}}(e) \ge d_{\tilde{\mathcal{B}}}(e) - 3\delta_{4} n$. 
    \end{claim}
    \begin{proof}[Proof of Claim~\ref{CLAIM:K43-Mcodeg-Bcode-K}]
        Suppose to the contrary that there exists $e = \{u_1, u_2\} \in K[V_1, V_2, V_3]$ such that this claim fails. 
        By symmetry, we may assume that $(u_1, u_2) \in V_1 \times V_2$. 
        
        Let $N_{2} \coloneqq N_{\mathcal{H}}(e) \cap V_{2}$ and $N_{3} \coloneqq N_{\mathcal{H}}(e) \cap V_{3}$. 
        We may assume that $\min\{|N_2|,~|N_3|\} > 2\delta_{4} n$. 
        Indeed, if $|N_2| \le 2\delta_{4} n$, then 
        \begin{align*}
            d_{\tilde{\mathcal{B}}}(e)
            \le d_{\mathcal{B}^{\mathrm{bi}}}(e)
            = |N_2|
            \le 2\delta_{4} n
            \le d_{\mathcal{M}^{\mathrm{tri}}}(e) + 3\delta_{4} n. 
        \end{align*}
        If $|N_3| \le 2\delta_{4} n$, then it follows from Claim~\ref{CLAIM:max-deg-B} that  
        \begin{align*}
            d_{\mathcal{M}^{\mathrm{tri}}}(e)
            = |V_3| - |N_3|
            & \ge \frac{n}{3} - 10\delta_{1}^{1/2} n - 2\delta_{4} n \\
            & \ge \frac{n}{3} + 10\delta_{1}^{1/2} n - 3\delta_{4} n
            \ge |V_2| - 3 \delta_{4} n
            \ge d_{\tilde{\mathcal{B}}}(e) - 3\delta_{4} n. 
        \end{align*}
        Therefore, we may assume that $\min\{|N_2|,~|N_3|\} > 2\delta_{4} n$. 
      
        Fix a vertex $y \in N_{2}$. 
        By the definition of $\tilde{\mathcal{B}}$, we have $d_{\mathcal{M}}(u_2y) \le \delta_{4} < |N_3|/2$. Therefore, 
        \begin{align*}
            |N_{3} \cap N_{\mathcal{H}}(u_2 y)|
            \ge |N_{3}| - d_{\mathcal{M}}(u_2y) 
            > \frac{|N_{3}|}{2}. 
        \end{align*}
        For every vertex $z\in N_{3} \cap N_{\mathcal{H}}(u_2 y)$, it follows from $K_{4}^{3}$-freeness of $\mathcal{H}$ that $u_1yz \in \mathcal{M}^{\mathrm{tri}}$. 
        Consequently, 
        \begin{align*}
            d_{\mathcal{M}}(u_1)
            \ge d_{\mathcal{M}^{\mathrm{tri}}}(u_1)
            \ge \frac{1}{2} |N_{2}||N_{3}|
            > \frac{1}{2} \cdot 2 \delta_{4} n \cdot 2 \delta_{4}n
            > 2\delta_{3} n, 
        \end{align*}
        a contradiction to Claim~\ref{CLAIM:max-deg-B}. 
    \end{proof}

    We now describe a two-phase procedure that transforms $\mathcal{H}$ into a subgraph of $\mathbb{C}[V_1, V_2, V_3]$.
    In Phase One, we remove all bad edges in $\mathcal{B}^{\mathrm{int}}$ and $\mathcal{B}^{\mathrm{bi}} \setminus \tilde{\mathcal{B}}$, while ensuring that the $\ell_{2}$-norm increases (using Lemma~\ref{LEMMA:K43-L2-improvement-phase-one}).
    In Phase Two, we remove all bad edges in $\tilde{\mathcal{B}} \subseteq \mathcal{B}^{\mathrm{bi}}$, again ensuring that the $\ell_{2}$-norm increases (using Lemma~\ref{LEMMA:K43-L2-improvement-phase-two}).
    This yields a contradiction, since it implies that $\norm{\mathcal{H}}_{2} < \norm{\mathbb{C}[V_1, V_2, V_3]}_{2}$, contradicting the maximality of $\mathcal{H}$.

    Define 
    \begin{align*}
        \mathcal{I}
        & \coloneqq \left\{e\in \overline{K}[V_1, V_2, V_3] \colon d_{\mathcal{M}}(e) \ge \delta_{4} n \right\},
        \quad\text{and}\quad \\[0.3em]
        \mathcal{J}
        & \coloneqq \left\{e\in K[V_1, V_2, V_3] \colon d_{\mathcal{M}^{\mathrm{tri}}}(e) \ge n/10 \right\}. 
    \end{align*}
    We label the elements of $\mathcal{I}$ and $\mathcal{J}$ as
    \begin{align*}
        \mathcal{I}
        = \{e_1, \ldots, e_{\ell}\}
        \quad\text{and}\quad 
        \mathcal{J}
        = \{\tilde{e}_1, \ldots, \tilde{e}_{m}\}. 
    \end{align*}

    \medskip 
    
    \textbf{Phase One}: 
    Let $\mathcal{H}_{0} \coloneqq \mathcal{H}$. 
    Suppose that $\mathcal{H}_{i}$ has been defined for some $i \in [0,\ell-1]$.
    We then let 
    \begin{align*}
        \mathcal{H}_{i+1}
        \coloneqq \big( \mathcal{H}_{i}  \setminus \mathcal{B}(e_{i+1}) \big) \cup \mathcal{M}(e_{i+1}). 
    \end{align*}
    That is, we remove all bad edges containing $e_{i+1}$ and insert all missing edges containing $e_{i+1}$ (see the figure on the left in Figure~\ref{fig:local-adjust}).

    For $i \in [0,\ell]$, let 
    \begin{align*}
        \mathcal{M}_{i}
        \coloneqq \mathbb{C}[V_1, V_2, V_3] \setminus \mathcal{H}_{i}
        \quad\text{and}\quad 
        \mathcal{B}_{i}
        \coloneqq \mathcal{H}_{i} \setminus \mathbb{C}[V_1, V_2, V_3]. 
    \end{align*}
    It follows from the definition of each $\mathcal{H}_{i}$ that 
    \begin{align*}
        \mathcal{B} = \mathcal{B}_{0} \supseteq \mathcal{B}_{1} \supseteq \cdots \supseteq \mathcal{B}_{\ell}
        \quad\text{and}\quad 
        \mathcal{M} = \mathcal{M}_{0} \supseteq \mathcal{M}_{1} \supseteq \cdots \supseteq \mathcal{M}_{\ell}. 
    \end{align*}

    By Claim~\ref{CLAIM:K43-bad-contains-heavy} and the definition of $\tilde{\mathcal{B}}$, every bad edge in $\mathcal{B}^{\mathrm{int}}\cup \big(\mathcal{B}^{\mathrm{bi}} \setminus \tilde{\mathcal{B}} \big)$ contains some element of $\mathcal{I}$. 
    Therefore, the final $3$-graph $\mathcal{H}_{\ell}$ contains no bad edge from $\mathcal{B}^{\mathrm{int}}\cup \big(\mathcal{B}^{\mathrm{bi}} \setminus \tilde{\mathcal{B}} \big)$.  

    \begin{claim}\label{CLAIM:K43-Phase-one-2norm-increase}
        We have $\norm{\mathcal{H}_{i+1}}_{2} > \norm{\mathcal{H}_{i}}_{2}$ for $i \in [0, \ell-1]$. 
    \end{claim}
    \begin{proof}[Proof of Claim~\ref{CLAIM:K43-Phase-one-2norm-increase}]
        Fix $i \in [0, \ell-1]$. 
        For every $j \in [0, \ell] \setminus \{i\}$, since $\{e_i, e_j\} \subseteq \overline{K}[V_1, V_2, V_3]$, we have $\mathcal{M}(e_i) \cap \mathcal{M}(e_j) = \emptyset$. 
        It follows that 
        \begin{align}\label{equ:K43-Phase-one-1}
            d_{\mathcal{M}_{i}}(e_{i+1}) 
            = d_{\mathcal{M}}(e_{i+1})
            \ge \delta_{4} n, 
        \end{align}
        Moreover, by Claim~\ref{CLAIM:K43-Mcodeg-Bcode-Kbar}, 
        \begin{align}\label{equ:K43-Phase-one-2}
            d_{\mathcal{M}_{i}}(e_{i+1}) 
            = d_{\mathcal{M}}(e_{i+1})
            \ge d_{\mathcal{B}}(e_{i+1}) - 7\delta_{3}^{1/2} n
            \ge d_{\mathcal{B}_{i}}(e_{i+1}) - 7\delta_{3}^{1/2} n. 
        \end{align}
        Applying Lemma~\ref{LEMMA:K43-L2-improvement-phase-one} to $\mathcal{H}_{i}$, with the parameter $\xi$ in that lemma taken to be $7\delta_{3}^{1/2}$, and noting that parts~\ref{LEMMA:K43-L2-improvement-phase-one-1}, \ref{LEMMA:K43-L2-improvement-phase-one-2}, \ref{LEMMA:K43-L2-improvement-phase-one-3},~\ref{LEMMA:K43-L2-improvement-phase-one-4}, and~\ref{LEMMA:K43-L2-improvement-phase-one-4} are verified by Claim~\ref{CLAIM:K43-L2-xi-size}, Claim~\ref{CLAIM:max-deg-B}, \eqref{equ:K43-Phase-one-1}, Claim~\ref{CLAIM:K43-B2-codegree-small}, and \eqref{equ:K43-Phase-one-2}, respectively, we conclude that $\norm{\mathcal{H}_{i+1}}_{2} > \norm{\mathcal{H}_{i}}_{2}$. 
    \end{proof}
    
    \medskip 
    
    \textbf{Phase Two}: 
    Let $\tilde{\mathcal{H}}_{0} \coloneqq \mathcal{H}_{\ell}$, that is, the final $3$-graph obtained at the end of Phase One. 
    Suppose that $\tilde{\mathcal{H}}_{i}$ has been defined for some $i \in [0, m-1]$.
    We then let 
    \begin{align*}
        \tilde{\mathcal{H}}_{i+1}
        \coloneqq \big( \tilde{\mathcal{H}}_{i} \setminus \mathcal{B}(\tilde{e}_{i}) \big) \cup \mathcal{M}^{\mathrm{tri}}(\tilde{e}_{i}). 
    \end{align*}
    For each $i \in [0, m]$, define  
    \begin{align*}
        \tilde{\mathcal{B}}_{i}
        \coloneqq \mathcal{B}_{\tilde{\mathcal{H}}_{i}}[V_1, V_2, V_3]
        \quad\text{and}\quad 
        \tilde{\mathcal{M}}_{i}
        \coloneqq \mathcal{M}_{\tilde{\mathcal{H}}_{i}}^{\mathrm{tri}}[V_1, V_2, V_3].
    \end{align*}
    By construction of the sequence $\tilde{\mathcal{H}}_{0}, \tilde{\mathcal{H}}_{1}, \ldots, \tilde{\mathcal{H}}_{m}$ and the property of $\mathcal{H}_{\ell}$, we have 
    \begin{align*}
        \tilde{\mathcal{B}} = \tilde{\mathcal{B}}_{0} \supseteq \tilde{\mathcal{B}}_{1} \supseteq \cdots \supseteq \tilde{\mathcal{B}}_{m}
        \quad\text{and}\quad 
        \tilde{\mathcal{M}}_{0} \supseteq \tilde{\mathcal{M}}_{1} \supseteq \cdots \supseteq \tilde{\mathcal{M}}_{m}. 
    \end{align*}
    It follows from Claim~\ref{CLAIM:K43-bad-contains-heavy} that every bad edge in $\tilde{\mathcal{B}}$ contains an element of $\mathcal{J}$. 
    Thus, the final $3$-graph $\tilde{\mathcal{H}}_{m}$ contains no edge from $\tilde{\mathcal{B}}$, and hence, no bad edge from $\mathcal{B}$.
    This means that $\tilde{\mathcal{H}}_{m} \subseteq \mathbb{C}[V_1, V_2, V_3]$. 
    
    \begin{claim}\label{CLAIM:K43-heavy-shadow-max-deg}
        We have $\Delta(\mathcal{J}) \le 40 \delta_{3} n$. 
    \end{claim}
    \begin{proof}[Proof of Claim~\ref{CLAIM:K43-heavy-shadow-max-deg}]
        Suppose to the contrary that this claim fails. 
        Let $v\in V$ be a vertex such that $d_{\mathcal{J}}(v) > 40 \delta_{3} n$. 
        It follows from the definitions of $\mathcal{J}$ that 
        \begin{align*}
            d_{\mathcal{M}}(v)
            \ge \frac{1}{2} \sum_{u \in N_{\mathcal{J}}(v)} d_{\mathcal{M}^{\mathrm{tri}}}(uv)
            \ge \frac{d_{\mathcal{J}}(v)}{2} \cdot \frac{n}{10}
            > 2\delta_{3} n^2, 
        \end{align*}
        a contradiction to Claim~\ref{CLAIM:max-deg-B}. 
    \end{proof}

    \begin{claim}\label{CLAIM:K43-phase-two-properties}
        For every $i \in [0,m-1]$, the $3$-graph $\tilde{\mathcal{H}}_{i}$ satisfies the following properties: 
        \begin{enumerate}[label=(\roman*)]
            \item\label{CLAIM:K43-phase-two-properties-1} $d_{\tilde{\mathcal{M}}_{i}}(\tilde{e}_{i+1}) \ge n/20$, 
            \item\label{CLAIM:K43-phase-two-properties-2} $d_{\tilde{\mathcal{M}}_{i}}(\tilde{e}_{i+1}) \ge d_{\tilde{\mathcal{B}}_{i}}(\tilde{e}_{i+1}) - 4\delta_{4} n$. 
        \end{enumerate}
    \end{claim}
    \begin{proof}[Proof of Claim~\ref{CLAIM:K43-phase-two-properties}]
        Fix $i \in [0,m-1]$, and suppose that $\tilde{e}_{i+1} = \{u_1, u_2\}$. 
        By symmetry, we may assume that $(u_1, u_2) \in V_{1} \times V_{2}$. 
        Let  
        \begin{align*}
            N \coloneqq N_{\tilde{\mathcal{M}}_{i}}(u_1 u_2) \subseteq V_{3}
            \quad\text{and}\quad 
            \hat{N} \coloneqq N_{\mathcal{M}^{\mathrm{tri}}}(u_1 u_2) \subseteq V_{3}. 
        \end{align*}
        Since $\tilde{\mathcal{M}}_{i} \subseteq \mathcal{M}^{\mathrm{tri}}$, we have $N \subseteq \hat{N}$. 
        For each $w \in \hat{N} \setminus N$, by the definition of Phase Two, the reason the triple $\{u_1, u_2, w\}$ is in $\mathcal{M}^{\mathrm{tri}}$ but not in $\tilde{\mathcal{M}}_{i}$ is that either $\{u_1, w\} = \tilde{e}_j$ or $\{u_2, w\} = \tilde{e}_j$ for some $j \le i-1$.
        Hence, by Claim~\ref{CLAIM:K43-heavy-shadow-max-deg},
        \begin{align*}
            |\hat{N} \setminus N|
            \le d_{\mathcal{J}}(u_1) + d_{\mathcal{J}}(u_2)
            \le 80\delta_{3} n. 
        \end{align*}
        It follows that 
        \begin{align*}
            d_{\tilde{\mathcal{M}}_{i}}(\tilde{e}_{i+1}) 
            = |N|
            \ge |\hat{N}| - 80\delta_{3}n
            = d_{\mathcal{M}^{\mathrm{tri}}}(\tilde{e}_{i+1}) - 80\delta_{3} n
            \ge \frac{n}{10} - 80\delta_{3}
            \ge \frac{n}{20}, 
        \end{align*}
        which proves part~\ref{CLAIM:K43-phase-two-properties-1}.
        
        Combining the inequality above with Claim~\ref{CLAIM:K43-Mcodeg-Bcode-K}, we obtain
        \begin{align*}
            d_{\tilde{\mathcal{M}}_{i}}(\tilde{e}_{i+1}) 
            & \ge d_{\mathcal{M}^{\mathrm{tri}}}(\tilde{e}_{i+1}) - 80\delta_{3} n \\
            & \ge d_{\tilde{\mathcal{B}}}(\tilde{e}_{i+1}) - 3\delta_{4} n - 80\delta_{3} n
            \ge d_{\tilde{\mathcal{B}}_{i}}(\tilde{e}_{i+1}) - 4\delta_{4} n, 
        \end{align*}
        which proves~\ref{CLAIM:K43-phase-two-properties-2}.
    \end{proof}

    For $i \in [0, m-1]$, applying Lemma~\ref{LEMMA:K43-L2-improvement-phase-two} to $\tilde{\mathcal{H}}_{i}$, with the parameter $\xi$ in that lemma taken to be $4\delta_{4}$, and noting that parts~\ref{LEMMA:K43-L2-improvement-phase-two-1}, \ref{LEMMA:K43-L2-improvement-phase-two-2}, \ref{LEMMA:K43-L2-improvement-phase-two-3}, and \ref{LEMMA:K43-L2-improvement-phase-two-4} are verified by Claim~\ref{CLAIM:K43-L2-xi-size}, Claim~\ref{CLAIM:max-deg-B}, Claim~\ref{CLAIM:K43-phase-two-properties}~\ref{CLAIM:K43-phase-two-properties-1}, and Claim~\ref{CLAIM:K43-phase-two-properties}~\ref{CLAIM:K43-phase-two-properties-2}, respectively, we conclude that $\norm{\tilde{\mathcal{H}}_{i+1}}_{2} > \norm{\tilde{\mathcal{H}}_{i}}_{2}$.
    Consequently, we have the strictly increasing sequence 
    \begin{align*}
        \norm{\mathcal{H}}_{2}
        = \norm{\mathcal{H}_{0}}_{2}
        < \norm{\mathcal{H}_{1}}_{2}
        < \cdots 
        < \norm{\mathcal{H}_{\ell}}_{2}
        = \norm{\tilde{\mathcal{H}}_{0}}_{2}
        < \cdots 
        < \norm{\tilde{\mathcal{H}}_{m}}_{2}
        \le \norm{\mathbb{C}[V_1, V_2, V_3]}, 
    \end{align*}
    which contradicts the maximality of $\mathcal{H}$. 
    This completes the proof of Theorem~\ref{THM:L2-exact-K43}. 
\end{proof}

\section{Concluding remarks}\label{SEC:Remarks}
Let $\mathbb{S}_{2}$ denote the unique $3$-graph on $4$ vertices with $2$ edges. 
Let $\mathrm{N}(\mathbb{S}_2, \mathcal{H})$ denote the number of copies of $\mathbb{S}_{2}$ in a $3$-graph $\mathcal{H}$. 
By definition, for every $3$-graph $\mathcal{H}$,
\begin{align*}
    \norm{\mathcal{H}}_2 
    = \sum_{e \in \binom{V(\mathcal{H})}{2}} d_{\mathcal{H}}^2(e) 
    = \sum_{e \in \binom{V(\mathcal{H})}{2}}\left(2 \tbinom{d_{\mathcal{H}}(e)}{2}+ d_{\mathcal{H}}(e)\right)
    = 2\mathrm{N}(\mathbb{S}_2, \mathcal{H}) + 3|\mathcal{H}|. 
\end{align*}
Thus, Theorem~\ref{THM:L2-exact-K43} implies that, for large $n$, the maximum of $\mathbb{N}(\mathbb{S}_{2}, \mathcal{H}) + \frac{3}{2} |\mathcal{H}|$ over all $n$-vertex $K_4^3$-free $3$-graphs $\mathcal{H}$ is uniquely attained by $\mathbb{C}_n$. 
A slight modification of our proof shows that $3/2$ can be replaced by any constant. 

\begin{theorem}\label{THM:K43-alpha-exact}
    For every real number $\alpha$, there exists $N_{0} = N_{0}(\alpha)$ such that the following holds for every $n \ge N_{0}$.
    Suppose that $\mathcal{H}$ is an $n$-vertex $K_{4}^{3}$-free $3$-graph. 
    Then 
    \begin{align*}
        \mathbb{N}(\mathbb{S}_{2}, \mathcal{H}) + \alpha |\mathcal{H}|
        \le \mathbb{N}(\mathbb{S}_{2}, \mathbb{C}_{n}) + \alpha |\mathbb{C}_{n}|, 
    \end{align*}
    with equality holds iff $\mathcal{H} \cong \mathbb{C}_{n}$. 
\end{theorem}

In particular, the case $\alpha = 0$ establishes exact bound for the generalized Tur\'{a}n number $\mathrm{ex}(n,\mathbb{S}_{2}, K_{4}^{3})$ for large $n$.

Given two disjoint sets $V_1$ and $V_2$, let $\mathbb{B}[V_1, V_2]$ denote the $3$-graph on $V_{1} \cup V_2$ consisting of all triples of types $V_1 V_1 V_2$ and $V_2 V_2 V_1$ (see Figure~\ref{fig:CB}), that is, 
\begin{align*}
    \mathbb{B}[V_1, V_2]
    \coloneqq \left\{e \colon \big(|e\cap V_1|, |e\cap V_2|\big) \in \{(2,1), (1,2)\}\right\}. 
\end{align*}
When $V_1  \cup V_2 = [n]$ is a balanced partition of $[n]$, we write $\mathbb{B}_{n}$ for $\mathbb{B}[V_1, V_2]$.

The construction $\mathbb{B}_{n}$ does not contain the complete $3$-graph on five vertices $K_{5}^{3}$. 
It follows that $\pi_{\ell_{2}}(K_{5}^{3}) \ge \frac{5}{8}$. 
Using computer-assisted flag algebra computations, Balogh--Clemen--Lidick\'{y}~\cite{BCL22b} proved that $\pi_{\ell_{2}}(K_{5}^{3}) \le \frac{5}{8}$, and thus, $\pi_{\ell_{2}}(K_{5}^{3}) = \frac{5}{8}$. 
They further conjectured (see {\cite[Conjecture~1.4]{BCL22b}}) that for large $n$, the construction $\mathbb{B}_{n}$ is the unique extremal construction for the Tur\'{a}n problem of $K_{5}^{3}$ in the $\ell_{2}$-norm. 
We will confirm this conjecture along with a special case of a related conjecture of Frankl--Gryaznov--Talebanfard~\cite{FGT22} in a forthcoming work, after establishing a Tur\'{a}n theorem in vertex-colored graphs.

\section*{Acknowledgments}
J.H. was supported by the National Key R\&D Program of China (No.~2023YFA1010202), the National Natural Science Foundation of China (No.~12071077), and the Central Guidance on Local Science and Technology Development Fund of Fujian Province (No.~2023L3003). 
X.L. was supported by the Excellent Young Talents Program (Overseas) of the National Natural Science Foundation of China. 
J.S. was supported by the Science and Technology Commission of Shanghai Municipality (No.~22DZ2229014) and the National Natural Science Foundation of China (No.~12331014).
\bibliographystyle{alpha}
\bibliography{K43L2Norm}

\newcommand{\etalchar}[1]{$^{#1}$}
\begin{thebibliography}{dCSdOFS16}

\bibitem[BCL22a]{BCL22a}
J\'{o}zsef Balogh, Felix~Christian Clemen, and Bernard Lidick\'{y}.
\newblock Hypergraph {T}ur\'{a}n problems in {$\ell_2$}-norm.
\newblock In {\em Surveys in combinatorics 2022}, volume 481 of {\em London Math. Soc. Lecture Note Ser.}, pages 21--63. Cambridge Univ. Press, Cambridge, 2022.

\bibitem[BCL22b]{BCL22b}
J\'ozsef Balogh, Felix~Christian Clemen, and Bernard Lidick\'y.
\newblock Solving {T}ur\'an's tetrahedron problem for the {$\ell_2$}-norm.
\newblock {\em J. Lond. Math. Soc. (2)}, 106(1):60--84, 2022.

\bibitem[BES74]{BES74}
B.~Bollob{\' a}s, P.~Erd{\H o}s, and E.~G. Straus.
\newblock Complete subgraphs of chromatic graphs and hypergraphs.
\newblock {\em Utilitas Math.}, 6:343--347, 1974.

\bibitem[Bro83]{Bro83}
W.~G. Brown.
\newblock On an open problem of {P}aul {T}ur\'an concerning {$3$}-graphs.
\newblock In {\em Studies in pure mathematics}, pages 91--93. Birkh\"auser, Basel, 1983.

\bibitem[BT11]{BT11}
Rahil Baber and John Talbot.
\newblock Hypergraphs do jump.
\newblock {\em Combin. Probab. Comput.}, 20(2):161--171, 2011.

\bibitem[CIL{\etalchar{+}}24]{CILLP24}
Wanfang Chen, Daniel Il'kovi{\v{c}}, Jared Le{\'o}n, Xizhi Liu, and Oleg Pikhurko.
\newblock {N}ondegenerate {T}ur\'{a}n problems under {$(t,p)$}-norms.
\newblock {\em arXiv preprint arXiv:2406.15934}, 2024.

\bibitem[CL99]{CL99}
Fan Chung and Linyuan Lu.
\newblock An upper bound for the {T}ur\'an number {$t_3(n,4)$}.
\newblock {\em J. Combin. Theory Ser. A}, 87(2):381--389, 1999.

\bibitem[CLY22]{CLY25}
Wanfang Chen, Changhong Lu, and Long-Tu Yuan.
\newblock A stability theorem for multi-partite graphs.
\newblock {\em arXiv preprint arXiv:2208.13995}, 2022.

\bibitem[dC88]{Caen88}
D.~de~Caen.
\newblock On upper bounds for {$3$}-graphs without tetrahedra.
\newblock volume~62, pages 193--202. 1988.
\newblock Seventeenth Manitoba Conference on Numerical Mathematics and Computing (Winnipeg, MB, 1987).

\bibitem[dC94]{Caen94Survey}
D.~de~Caen.
\newblock The current status of {T}ur{\'a}n's problem on hypergraphs.
\newblock In {\em Extremal problems for finite sets ({V}isegr\'ad, 1991)}, volume~3 of {\em Bolyai Soc. Math. Stud.}, pages 187--197. J\'anos Bolyai Math. Soc., Budapest, 1994.

\bibitem[dCSdOFS16]{SFS16}
M.~K. de~Carli~Silva, F.~M\'{a}rio de~Oliveira~Filho, and C.~M. Sato.
\newblock Flag algebras: a first glance.
\newblock {\em Nieuw Arch. Wiskd. (5)}, 17:193--199, 2016.

\bibitem[Erd81]{Erd81}
P.~Erd{\H o}s.
\newblock Problems and results in graph theory.
\newblock In {\em The theory and applications of graphs ({K}alamazoo, {M}ich., 1980)}, pages 331--341. Wiley, New York, 1981.

\bibitem[ES46]{ES46}
P.~Erd{\H o}s and A.~H. Stone.
\newblock On the structure of linear graphs.
\newblock {\em Bull. Amer. Math. Soc.}, 52:1087--1091, 1946.

\bibitem[ES66]{ES66}
P.~Erd{\H o}s and M.~Simonovits.
\newblock A limit theorem in graph theory.
\newblock {\em Studia Sci. Math. Hungar.}, 1:51--57, 1966.

\bibitem[FDF88]{Fon88}
D.~G. Fon-Der-Flaass.
\newblock A method for constructing {$(3,4)$}-graphs.
\newblock {\em Mat. Zametki}, 44(4):546--550, 559, 1988.

\bibitem[FGT22]{FGT22}
Peter Frankl, Svyatoslav Gryaznov, and Navid Talebanfard.
\newblock A variant of the {VC}-dimension with applications to depth-$3$ circuits.
\newblock In {\em 13th {I}nnovations in {T}heoretical {C}omputer {S}cience {C}onference}, volume 215 of {\em LIPIcs. Leibniz Int. Proc. Inform.}, pages Art. No. 72, 19. Schloss Dagstuhl. Leibniz-Zent. Inform., Wadern, 2022.

\bibitem[Fro08]{Fro08}
Andrew Frohmader.
\newblock More constructions for {T}ur\'an's {$(3,4)$}-conjecture.
\newblock {\em Electron. J. Combin.}, 15(1):Research Paper 137, 23, 2008.

\bibitem[F{\"u}r91]{Fur91}
Z.~F{\"u}redi.
\newblock Tur{\'a}n type problems.
\newblock In {\em Surveys in combinatorics, 1991 ({G}uildford, 1991)}, volume 166 of {\em London Math. Soc. Lecture Note Ser.}, pages 253--300. Cambridge Univ. Press, Cambridge, 1991.

\bibitem[GGH{\etalchar{+}}22]{GilboaGlebovHefetzLinialMorgenstein22}
Shoni Gilboa, Roman Glebov, Dan Hefetz, Nati Linial, and Avraham Morgenstern.
\newblock On the local structure of oriented graphs---a case study in flag algebras.
\newblock {\em Electron. J. Combin.}, 29(3):Paper No. 3.39, 53, 2022.

\bibitem[GLMP25]{GLMP24}
Jun Gao, Xizhi Liu, Jie Ma, and Oleg Pikhurko.
\newblock Phase transition of degenerate {T}ur\'an problems in {$p$}-norms.
\newblock {\em SIAM J. Discrete Math.}, 39(3):1712--1736, 2025.

\bibitem[HLZ25]{HLZ25Fano}
Jianfeng Hou, Xizhi Liu, and Yixiao Zhang.
\newblock Exact {T}ur{\' a}n number of the {F}ano plane in the $\ell_2$-norm.
\newblock {\em arXiv preprint arXiv:2507.12354}, 2025.

\bibitem[Kee11]{Kee11}
Peter Keevash.
\newblock Hypergraph {T}ur\'an problems.
\newblock In {\em Surveys in combinatorics 2011}, volume 392 of {\em London Math. Soc. Lecture Note Ser.}, pages 83--139. Cambridge Univ. Press, Cambridge, 2011.

\bibitem[Kos82]{Kos82}
A.~V. Kostochka.
\newblock A class of constructions for {T}ur{\'{a}}n's {$(3,\,4)$}-problem.
\newblock {\em Combinatorica}, 2(2):187--192, 1982.

\bibitem[LM22]{LM22stab}
Xizhi Liu and Dhruv Mubayi.
\newblock A hypergraph {T}ur{\'a}n problem with no stability.
\newblock {\em Combinatorica}, 42(3):433--462, 2022.

\bibitem[Man07]{Mantel07}
W.~Mantel.
\newblock Vraagstuk {XXVIII}.
\newblock {\em Wiskundige Opgaven}, 10(2):60--61, 1907.

\bibitem[Raz07]{Raz07}
Alexander~A. Razborov.
\newblock Flag algebras.
\newblock {\em J. Symbolic Logic}, 72(4):1239--1282, 2007.

\bibitem[Raz10]{Razborov10}
Alexander~A. Razborov.
\newblock On 3-hypergraphs with forbidden 4-vertex configurations.
\newblock {\em SIAM J. Discrete Math.}, 24(3):946--963, 2010.

\bibitem[RS78]{RS78}
I.~Z. Ruzsa and E.~Szemer{\'e}di.
\newblock Triple systems with no six points carrying three triangles.
\newblock In {\em Combinatorics ({P}roc. {F}ifth {H}ungarian {C}olloq., {K}eszthely, 1976), {V}ol. {II}}, volume~18 of {\em Colloq. Math. Soc. J\'anos Bolyai}, pages 939--945. North-Holland, Amsterdam-New York, 1978.

\bibitem[Sid95]{Sid95}
A.~Sidorenko.
\newblock What we know and what we do not know about {T}ur{\'a}n numbers.
\newblock {\em Graphs Combin.}, 11(2):179--199, 1995.

\bibitem[Sim68]{Sim68}
M.~Simonovits.
\newblock A method for solving extremal problems in graph theory, stability problems.
\newblock In {\em Theory of {G}raphs ({P}roc. {C}olloq., {T}ihany, 1966)}, pages 279--319. Academic Press, New York-London, 1968.

\bibitem[Tur41]{Tur41}
Paul Tur\'an.
\newblock Eine {E}xtremalaufgabe aus der {G}raphentheorie.
\newblock {\em Mat. Fiz. Lapok}, 48:436--452, 1941.

\bibitem[Zyk49]{Zyk49}
A.~A. Zykov.
\newblock On some properties of linear complexes.
\newblock {\em Mat. Sbornik N.S.}, 24(66):163--188, 1949.

\end{thebibliography}
\begin{appendix}
\section*{Appendix}
\begin{proof}[Proof of Fact~\ref{FACT:Cn-L2-max}]
    Let $n_i \coloneqq |V_i|$ for $i \in [3]$. 
    It is enough to show that $|n_i - n_j| \le 1$ for all $i, j \in [3]$. 
    Suppose on the contrary that this is false. 
    We will construct a $3$-graph $\mathbb{C}[V_1', V_2', V_3']$ with $V_1' \cup V_2' \cup V_3' = [n]$ such that $\norm{\mathbb{C}[V_1', V_2', V_3']}_{2} > \norm{\mathbb{C}[V_1, V_2, V_3]}_{2}$, thus proving the lemma. 
    Without loss of generality, we may assume that $n_1 = \max\{n_1,n_2,n_3\}$. 
    Since $n_1+n_2+n_3 = n \ge 6$, we have $n_1 \ge 2$ by average. 
    By symmetry, there are two cases. 
    
    \textbf{Case 1:} $n_1 \ge n_2 \ge n_3$ and $n_1 \ge n_3 +2$. 
    
    Let $|V_1'| = n_1 -1$, $|V_2'| = n_2$, and $|V_3'| = n_3+1$. 
    Some calculations show that 
    \begin{align}\label{EQ:V1'-V1-expression1}
    \norm{\mathbb{C}[V_1', V_2', V_3']}_{2} - \norm{\mathbb{C}[V_1, V_2, V_3]}_{2} 
    = &  n_1^2 n_3 + 2n_1(n_2^2+n_3^2-n_2-n_3) + n_2^3 - n_3^3 - 2 n_2^2 n_3 \notag \\
    & - 3 n_2 n_3^2 - \frac{7}{2} n_2^2 - \frac{1}{2}n_3^2 + 2 n_2 n_3 +\frac{5}{2}n_2 + \frac{1}{2}n_3. 
    \end{align}
    It suffices to show that this expression is strictly positive. 
    When $n_3 = 0$, this is a non-decreasing linear function of $n_1$. 
    When $n_3 \ge 1$, this is a quadratic function of $n_1$, and it is increasing in $n_1 \ge 2$. 
    So the minimum value of~\eqref{EQ:V1'-V1-expression1} is attained when $n_1 = \max\{n_2, n_3+2\}$. 
    
    If $n_1 = n_2$, then it follows from~\eqref{EQ:V1'-V1-expression1} that 
    \begin{align*}
    \norm{\mathbb{C}[V_1', V_2', V_3']}_{2} - \norm{\mathbb{C}[V_1, V_2, V_3]}_{2} 
    = 3 n_1^3 -\frac{11}{2}n_1^2 - n_3 n_1^2 + \frac{5}{2}n_1 - n_3^2 n_1 -n_3^3 - \frac{1}{2}n_3^2+\frac{1}{2}n_3,  
    \end{align*}
    which is increasing in $n_1 \ge n_3 +2$. 
    So it is enough to assume that $n_1 =n_2 =n_3+2$. 
    Then we have $\norm{\mathbb{C}[V_1', V_2', V_3']}_{2} - \norm{\mathbb{C}[V_1, V_2, V_3]}_{2}  = 6 n_3^2 + 13 n_3 + 7 > 0$. 
    
    If $n_1 = n_3 +2$, then $n_2$ is one of $n_3 +2, n_3 +1$ and $n_3$. 
    Substituting it into~\eqref{EQ:V1'-V1-expression1}, then $\norm{\mathbb{C}[V_1', V_2', V_3']}_{2} - \norm{\mathbb{C}[V_1, V_2, V_3]}_{2}$ is $6n_3^2 + 13 n_3 + 7$, $6 n_3^2 + 3 n_3 $ and $6 n_3^2 - n_3$, respectively. 
    The assumptions $n_1 = n_3 +2$ and $n \ge 6$ imply that $n_3 \ge 1$. 
    Therefore, we have $\norm{\mathbb{C}[V_1', V_2', V_3']}_{2} - \norm{\mathbb{C}[V_1, V_2, V_3]}_{2} > 0$, as desired. 
    
    \textbf{Case 2:} $n_1 \ge n_3 \ge n_2$ and $n_1 \ge n_2 +2$.
    
    Let $|V_1'|=n_1 - 1$, $|V_2'| = n_2+1$, and $|V_3'| = n_3$. 
    Some calculations show that 
    \begin{align}\label{EQ:V1'-V1-expression2}
        \norm{\mathbb{C}[V_1', V_2', V_3']}_{2} - \norm{\mathbb{C}[V_1, V_2, V_3]}_{2} 
        = & \left(1- n_1\right) n_2^2 -2 n_2\left(n_1^2 + n_3^2 - 2n_1 - n_3 + 1\right) + n_1^3- \frac{9}{2}n_1^2 \notag \\
        & + \frac{13}{2}n_1 - n_3^3-\frac{1}{2}n_3^2+ \frac{9}{2}n_3+ 3 n_1^2 n_3  + 2 n_1 n_3^2- 8 n_1 n_3 -3. 
    \end{align}
    This is a quadratic function of $n_2$, and it is decreasing in $n_2 \ge 0$. 
    So it is enough to consider the case $n_2 = \min\{n_3,n_1-2\}$. 
    
    If $n_2 = n_3$, then it follows from~\eqref{EQ:V1'-V1-expression2} that 
    \begin{multline}
    \norm{\mathbb{C}[V_1', V_2', V_3']}_{2} - \norm{\mathbb{C}[V_1, V_2, V_3]}_{2} \notag \\
    = n_1^3+ n_2 n_1^2- \frac{9}{2}n_1^2+ n_2^2 n_1 + \frac{13}{2}n_1- 4 n_1 n_2- 3 n_2^3+ \frac{5}{2}n_2^2 + \frac{5}{2}n_2 -3, \notag
    \end{multline}
    which is increasing in $n_1 \ge n_2+2$. 
    So it is enough to assume that $n_1 = n_2 +2$ and $n_3 = n_2$. 
    Then we have $\norm{\mathbb{C}[V_1', V_2', V_3']}_{2} - \norm{\mathbb{C}[V_1, V_2, V_3]}_{2} = 6 n_2^2 - n_2 > 0$. 
    
    If $n_2 = n_1 -2$, then $n_3$ is one of $n_1$, $n_1 -1$ and $n_1 -2$. 
    Substituting it into~\eqref{EQ:V1'-V1-expression2}, then $\norm{\mathbb{C}[V_1', V_2', V_3']}_{2} - \norm{\mathbb{C}[V_1, V_2, V_3]}_{2}$ is $6 n_1^2 -11 n_1 +5$, $6n_1^2-15n_1+9$ and $6 n_1^2 -25 n_1+26$, respectively. 
    The assumptions $n_1 \ge n_2 + 2$ and $n \ge 6$ imply that $n_1 \ge 3$. 
    Therefore, we have $\norm{\mathbb{C}[V_1', V_2', V_3']}_{2} - \norm{\mathbb{C}[V_1, V_2, V_3]}_{2} > 0$, as desired. 
\end{proof}

\end{appendix}

\newpage 
\section*{\normalfont Authors}

\begin{multicols}{2}
\begin{flushleft}

\vbox{%
Levente Bodn\'ar \\
{\small Mathematics Institute and DIMAP} \\
{\small University of Warwick} \\
{\small Coventry, CV4 7AL, UK} \\
\texttt{bodnalev@gmail.com}
}

\vspace{0.7cm}

\vbox{%
Wanfang Chen \\
{\small School of Mathematical Sciences} \\
{\small University of Science and Technology of China} \\
{\small Hefei, 230026, China} \\
\texttt{a372959313@gmail.com}
}

\vspace{0.7cm}

\vbox{%
Jinghua Deng \\
{\small Center for Discrete Mathematics} \\
{\small Fuzhou University} \\
{\small Fujian, 350003, China} \\
\texttt{Jinghua\_deng@163.com}
}

\vspace{0.7cm}

\vbox{%
Jianfeng Hou \\
{\small Center for Discrete Mathematics} \\
{\small Fuzhou University} \\
{\small Fujian, 350003, China} \\
\texttt{jfhou@fzu.edu.cn}
}

\vspace{0.7cm}

\vbox{%
Xizhi Liu \\
{\small School of Mathematical Sciences} \\
{\small University of Science and Technology of China} \\
{\small Hefei, 230026, China} \\
\texttt{liuxizhi@ustc.edu.cn}
}

\vspace{0.7cm}

\vbox{%
Jialei Song \\
{\small School of Mathematical Sciences} \\
{\small Key Laboratory of MEA (Ministry of Education) \& Shanghai Key Laboratory of PMMP} \\
{\small East China Normal University} \\
{\small Shanghai 200241, China} \\
\texttt{jialsong@foxmail.com}
}

\vspace{0.7cm}

\vbox{%
Jiabao Yang \\
{\small School of Mathematics} \\
{\small Nanjing University} \\
{\small Nanjing, 210093, China} \\
\texttt{jbyang1215@nju.edu.cn}
}

\vspace{0.7cm}

\vbox{%
Yixiao Zhang \\
{\small Center for Discrete Mathematics} \\
{\small Fuzhou University} \\
{\small Fujian, 350003, China} \\
\texttt{fzuzyx@gmail.com}
}

\end{flushleft}
\end{multicols}

\end{document}